\documentclass[a4paper,11pt]{amsart}

\usepackage[top = 3cm, bottom = 3cm, left=3cm, right=3cm]{geometry}
\usepackage[utf8]{inputenc}

\usepackage{amsmath}
\usepackage{amssymb}
\usepackage{mathdots}
\usepackage{enumerate}
\usepackage{color}
\usepackage{appendix}
\usepackage{graphicx}

\usepackage{mathdots}
\usepackage{enumerate}
\usepackage{appendix}
\usepackage{listings}
\usepackage{physics}
\usepackage{tikz}
\usepackage{tcolorbox}
\usepackage{amsfonts}

\usepackage{hyperref}
\hypersetup{
    colorlinks,
    linkcolor={red!50!black},
    citecolor={blue!50!black},
    urlcolor={blue!80!black}
}

\newtheorem{theorem}{Theorem}[section]

\newtheorem{proposition}[theorem]{Proposition}
\newtheorem{lemma}[theorem]{Lemma}

\theoremstyle{definition}

\newtheorem{definition}[theorem]{Definition}

\newtheorem{remarkx}[theorem]{Remark}
\newenvironment{remark}
  {\pushQED{\qed}\remarkx}
  {\popQED\endremarkx}

\newcommand{\eps}{\epsilon}
\newcommand{\thickdot}{\text{\scalebox{0.7}{$\blacksquare$}}}

\renewcommand{\bar}[1]{\overline{#1}}

\DeclareMathOperator{\GL}{GL}
\DeclareMathOperator{\Hom}{Hom}

\DeclareMathOperator{\mult}{mult}
\DeclareMathOperator{\codim}{codim}
\newcommand{\bbP}{\mathbb{P}}
\newcommand{\bbC}{\mathbb{C}}
\newcommand{\calO}{\mathcal{O}}
\newcommand{\id}{\mathrm{id}}
\newcommand{\Str}{\mathrm{Str}}

\newcommand{\vvirg}{, \ldots ,}

\usepackage{MnSymbol}
\usepackage{graphicx,pict2e, picture}

\newcommand{\bgeq}{
  \mathrel{
    \vphantom{\geq}
    \mathpalette{\bgleqinn\blacktriangleright}{0.25}
  }
}

\makeatletter
\newcommand{\bgleqinn}[3]{
  \sbox\z@{$#2\m@th#1$}
  \linethickness{.1\ht\z@}
  \begin{picture}(\wd\z@,\ht\z@)(0,-.15\ht\z@)
  \roundcap
  \put(#3\wd\z@,-.2\ht\z@){\line(1,0){.65\wd\z@}}
  \put(0,0){\box\z@}
  \end{picture}
}

\renewcommand{\Im}{\mathrm{Im}}

\title{Partial degeneration of tensors}

\author[M. Christandl]{Matthias Christandl}
\address[M. Christandl]{QMATH, Department of Mathematical Sciences, University of Copenhagen, Universitetsparken 5, 2100 Copenhagen, Denmark}
 \email[Christandl]{christandl@math.ku.dk}

\author[F. Gesmundo]{Fulvio Gesmundo}
 \address[F. Gesmundo]{Saarland Informatics Campus, Universit\"at des Saarlandes, 66123 Saarbr\"ucken, Germany; (current) Institut de Mathématiques de Toulouse, Université Paul Sabatier, Toulouse, France.}
 \email{fgesmund@math.univ-toulouse.fr}

\author[V. Lysikov]{Vladimir Lysikov}
 \address[V. Lysikov]{Faculty of Computer Science, Ruhr University Bochum, Universit\"{a}tstra\ss{}e 150, 44801 Bochum, Germany}
 \email{vladimir.lysikov@rub.de}

\author[V. Steffan]{Vincent Steffan}
\address[V. Steffan]{QMATH, Department of Mathematical Sciences, University of Copenhagen, Universitetsparken 5, 2100 Copenhagen, Denmark}
\email[Steffan]{sv@math.ku.dk}
 
\keywords{Restriction, Degeneration, Tensor Rank}
\subjclass[2020]{(primary) 15A69; (secondary) 14N07, 68Q15, 81P45}

\begin{document}
\begin{abstract}
Tensors are often studied by introducing preorders such as \textit{restriction} and \emph{degeneration}: the former describes transformations of the tensors by \emph{local} linear maps on its tensor factors; the latter describes transformations where the local linear maps may vary along a curve, and the resulting tensor is expressed as a limit along this curve. In this work we introduce and study \emph{partial degeneration}, a special version of degeneration where one of the local linear maps is constant whereas the others vary along a curve. Motivated by algebraic complexity, quantum entanglement and tensor networks, we present constructions based on matrix multiplication tensors and find examples by making a connection to the theory of prehomogeneous tensor spaces. We highlight the subtleties of this new notion by showing obstruction and classification results for the unit tensor. To this end, we study the notion of aided rank, a natural generalization of tensor rank. The existence of partial degenerations gives strong upper bounds on the aided rank of a tensor, which allows one to turn degenerations into restrictions. In particular, we present several examples, based on the W-tensor and the Coppersmith-Winograd tensors, where lower bounds on aided rank provide obstructions to the existence of certain partial degenerations. 
\end{abstract}
\maketitle

\section{Introduction}\label{sec:Introduction}

Restriction and degeneration are preorders on the set of tensors that capture many important concepts in classical algebraic geometry, complexity theory, combinatorics, entanglement theory and the study of tensor networks. In this work, we introduce the notion of \emph{partial degeneration}, a special version of degeneration defining a preorder which is intermediate between restriction and degeneration.

One key contribution of this work is to show that restriction, partial degeneration and degeneration are mutually inequivalent notions. To show a separation between the notions of restriction and partial degeneration we present a number of constructions based on tensors motivated from algebraic complexity theory and tensor networks. We also draw a connection to the theory of prehomogeneous tensor spaces which allows us to derive further examples manifesting this separation. To show a separation between partial degenerations and degenerations, we prove a no-go result for the unit tensor which moreover allows us to classify certain families of partial degenerations. 

Moreover, we introduce the notion of \emph{aided restriction}, which is performed on a version of the tensor augmented via an \textit{aiding matrix}. This raises the question on how large the rank of such an aiding matrix should be in order to allow certain restrictions. We study upper and lower bounds, highlighting the role of degenerations and partial degeneration.

In the remainder of this introduction we describe in more detail the main contributions of this work and briefly outline future directions and open questions.

\subsection{Background}\label{subsec:Background}

Throughout the paper, $U_1,U_2,U_3,V_1,V_2,V_3$ are complex vector spaces of dimension $u_1,u_2,u_3,v_1,v_2,v_3$, respectively. For each of these spaces, write $e_1 ,e_2 , \ldots $ for a fixed basis. Tensors in $U_1 \otimes U_2 \otimes U_3$ can be understood as a resources by introducing the notions of restriction and degeneration: For tensors $T \in U_1 \otimes  U_2 \otimes  U_3$ and $S \in V_1\otimes  V_2 \otimes  V_3$, we say that $T$ \emph{restricts} to $S$, and write $T \geq S$, if there exist linear maps $A_i : U_i \rightarrow V_i$ for $i = 1,2,3$ such that 
\[
S = (A_1 \otimes A_2 \otimes A_3)T;
\]
we say that $T$ \emph{degenerates} to $S$, and write $T \trianglerighteq S$, if $S$ is a limit of restrictions of $T$.  It is a classical result that $T \trianglerighteq S$ if and only if there are linear maps $A_i(\epsilon) : U_i \rightarrow V_i$ depending polynomially on $\epsilon$, that is with entries in the polynomial ring $\mathbb{C}[\epsilon]$, such that 
\[
(A_1(\epsilon) \otimes A_2(\epsilon) \otimes A_3(\epsilon))T =  \epsilon^d S + \epsilon^{d+1}S_1 + \dots  + \epsilon^{d+e}S_e
\]
for some natural numbers $d,e$ and some tensors $S_1 \vvirg S_e$. The integers $d$ and $e$ are called \emph{approximation degree} and \emph{error degree}, respectively; write $T \trianglerighteq_d^e S$ when it is useful to keep track of these integers. 

One can define analogous notions for any number $k$ of tensor factors. It is known that for $k =2$, the notions of restriction and degeneration are equivalent whereas for $k = 3$ they differ. This phenomenon already appears in \cite{Sylv:PrinciplesCalculusForms} and it occurs already when the tensor factors are two 2-dimensional. Write $\langle r \rangle = e_1 \otimes e_1 \otimes e_1 + \dots + e_r \otimes  e_r \otimes e_r$ for the $r$-th \emph{unit tensor} in $\mathbb{C}^{r}\otimes \mathbb{C}^{r}\otimes \mathbb{C}^{r}$. Define the $W$-tensor in $\mathbb{C}^{2}\otimes \mathbb{C}^{2}\otimes \mathbb{C}^{2}$ to be $W = e_1 \otimes  e_1 \otimes e_2 + e_1 \otimes e_2 \otimes e_1 + e_2 \otimes e_1 \otimes e_1$. It is a classical result that $ \langle 2 \rangle \trianglerighteq W$ but $\langle 2 \rangle \ngeq W$. 

Understanding tensors in this resource theoretic way has led to advances in various fields in mathematics, physics and computer science: 
\begin{sloppy}
\begin{itemize}
    \item The study of degenerations has lead to advances in several combinatorial problems, in particular the sunflower problem and cap sets~\cite{10.4007/annals.2017.185.1.8,tao}. In~\cite{christandl_et_al:LIPIcs.ITCS.2022.48}, the study of \textit{combinatorial degenerations} has lead to advances in the problem of finding large corner free sets.
    \item In quantum information theory, the study of restrictions and degenerations led to an improved understanding of the entanglement of multi-partite quantum states~\cite{threequbitsthreedifferentways,ChtDjo:NormalFormsTensRanksPureStatesPureQubits}.
    \item \textit{Tensor network representations} of many body quantum states are a special case of tensor restrictions. Studying degenerations of tensors can lead to more efficient tensor network representations of quantum states~\cite{tensornetworkfromgeometry,ChrGesStWer:Optimization}.
    \item In algebraic complexity theory, tensor restriction and degeneration play an important role in the study of the asymptotic complexity of bilinear maps. In particular, starting from \cite{BiCaLoRo:O277ComplexityApproximateMatMult,Bini:RelationsExactApproxBilAlg}, essentially all upper bounds on the exponent of the matrix multiplication rely on degenerations of suitable tensors to the matrix multiplication tensor \cite{STRASSEN+1987+406+443,10.1145/28395.28396,AlmWil:RefinedLaserMethodFMM}. A refined understanding of these notions, even in small cases, can lead to further improvements \cite{1997-buergisser,gs005,ConGesLanVen:RankBRankKron}. 
\end{itemize}
\end{sloppy}

In of all these cases, the  notions of restriction and degeneration are used to compare tensors: For example, in complexity theory, $T \geq S$ or $T \trianglerighteq S$ reflects the fact that the bilinear map corresponding to $T$ is harder to compute than the one corresponding to $S$; in quantum physics, it expresses the fact that the quantum state described by $T$ is more entangled than the one described by $S$.

\subsection{Partial degeneration}\label{subsec:introPartial degenerations}

In~\autoref{sec:Degenerations with one constant map}  of this work, we introduce and study the intermediate notion of \textit{partial degeneration}: For tensors $T \in U_1 \otimes U_2 \otimes U_3$ and $S \in V_1 \otimes  V_2 \otimes V_3$, we say that $T$ partially degenerates to $S$, and write $ T \bgeq S$, if $T$ degenerates to $S$ where the degeneration map $A_1(\epsilon ) = A_1$ can be chosen constant in $\epsilon$. Analogous notions can be defined by assuming that $A_2(\epsilon)$ or $A_3(\epsilon)$ are constant. For simplicity, we will always assume that the constant map is the one acting on the first factor. Hence, we have $T\bgeq S$ if and only if there are linear maps $A_1, A_2 (\epsilon)$ and $A_3(\epsilon)$, with $A_2(\epsilon)$ and $A_3(\epsilon)$ depending polynomially on $\epsilon$ such that $\lim_{\epsilon \rightarrow 0} \frac{1}{\epsilon^d}(A_1 \otimes A_2(\epsilon) \otimes A_3(\epsilon)) T = S$ for some $d$. As in the case of degeneration, we sometimes write $T \bgeq _d^e S$ to keep track of the approximation and error degrees.  We point out that requiring two of the three linear maps to be constant in $\eps$ provides a notion of degeneration which is equivalent to restriction, see~\autoref{rmk:one moving map}. A priori, it is unclear whether the notion of partial degeneration is indeed non-trivial, or whether one might always reduce a degeneration to a partial degeneration or a partial degeneration to a restriction. We will show this is not the case. We point out that an example of partial degeneration has been known since \cite{STRASSEN+1987+406+443}, and it was used to achieve a breakthrough result in the study of the complexity of matrix multiplication: the tensor 
\begin{equation}\label{eq:strasseninintro}
    \text{Str}_q = \sum_{i = 1}^{q-1} e_i \otimes e_q \otimes e_i + e_i \otimes e_i \otimes e_q \in \mathbb{C}^{q-1} \otimes \mathbb{C}^{q} \otimes \mathbb{C}^{q}
\end{equation}
has tensor rank equal to $2q - 2$ but it is a partial degeneration of the unit tensor $\langle q  \rangle$.

In this work, we study for the first time partial degenerations in depth. In~\autoref{sec:Degenerations with one constant map}, we construct various families of examples of partial degenerations which are not restrictions. We also study the question under which circumstances partial degenerations cannot exist and provide a no-go result for certain partial degenerations of the unit tensor.

In~\autoref{subsec:Nullsubtensors}, we construct a family of partial degenerations of the matrix multiplication tensor. Let $\langle m,n,p \rangle$ be the matrix multiplication tensor associated to the bilinear map multiplying an $m \times n$ matrix with a $n \times p$ matrix. To construct a family of partial degenerations of the matrix multiplication tensor $\langle 2,2,2 \rangle$, the challenge is to show that these are not actually restrictions of $\langle 2,2,2 \rangle$. To see this, we resort to notion of tensor compressibility, in the sense of~\cite{DBLP:journals/corr/LandsbergM16a}. 

In~\autoref{subsec:Matrix pencils}, we study the notion of partial degeneration in the setting of prehomogeneous tensor spaces where we can find many more examples of honest partial degenerations. We say that the action of a linear algebraic group $G$ on a vector space $V$ is \textit{prehomogeneous} if there is an element $v\in V$ whose orbit under the group is dense in $V$, equivalently in the Zariski or Euclidean topology. Consider the action of $\GL(U_2) \times \GL(U_3)$ on $U_1 \otimes U_2 \otimes U_3$; if this action is prehomogeneous, with the tensor $T$ having a dense orbit, then every tensor in $U_1 \otimes U_2 \otimes U_3$ is a partial degeneration of $T$. Prehomogeneous tensor spaces of this form have been studied in~\cite{SatKim:ClassificationIrredPrehomVS} and are well-understood, and the prehomogeneity of the action is determined by simple arithmetic relations among the dimensions of the tensor factors. In~\autoref{subsec:Matrix pencils}, for every instance where the space $U_1 \otimes U_2 \otimes U_3$ is prehomogeneous under the action of $\GL(U_2) \otimes \GL(U_3)$, we provide an example of a tensor that cannot be a restriction of a tensor with dense orbit. We emphasize that while it is well-understood under which conditions $U_1 \otimes U_2 \otimes U_3$ is prehomogeneous, there are in general no closed formulas for elements having dense orbit. If $\dim (U_1) = 2$, that is the case of \textit{matrix pencils}, explicit elements with dense orbit are known, see, e.g., \cite[Ch. XIII]{Gant:TheoryOfMatrices} and \cite{Pok:PerturbationsEquivalenceOrbitMatrixPencil}. In~\autoref{subsec:Matrix pencils}, we use these examples to provide explicit partial degenerations for matrix pencils.

In~\autoref{subsec:A no-go result for partial degenerations}, we study situations in which partial degenerations cannot occur. More precisely, we show that every partial degeneration of the unit tensor $\langle r \rangle$ to a concise tensor $T \in U_1 \otimes U_2 \otimes U_3$ with $\dim (U_1) =r$ can be reduced to a restriction. We use this result to show that for $\dim(U_1) = r -1$, tensors of the form~\autoref{eq:strasseninintro} are essentially all honest partial degenerations that can occur. Furthermore, we construct honest partial degenerations of $\langle r \rangle$ for any $r$ to non-concise tensors. These no-go results show that restriction, partial degenerations and degenerations are in fact three different notions.

\subsection{Aided restriction and aided rank}\label{subsec:introAided rank}

The starting point of the second part of this work is the fact that any degeneration can be turned into a restriction using \textit{interpolation}. It is known that if $T \trianglerighteq^e_d S$ then $T \boxtimes \langle 2d + 1 \rangle \geq S$ and $T \boxtimes \langle e+1 \rangle \geq S$, where $\boxtimes$ is the Kronecker product of tensors; this dates back essentially to \cite{BiCaLoRo:O277ComplexityApproximateMatMult}, and extends to a number of different settings \cite{ChrGesJen:BorderRankNonMult,tensornetworkfromgeometry,ChrGesStWer:Optimization}. In~\autoref{sec:Aided rank}, we study the case where the supporting tensor is a matrix instead of a unit tensor. More precisely, for a tensor $T \in U_1 \otimes U_2 \otimes U_3$, define 
\begin{equation*}
    T^{\thickdot p} = T \boxtimes \langle 1,1,p\rangle \in U_1 \otimes (U_2 \otimes \bbC^p) \otimes (U_3 \otimes \bbC^p);
\end{equation*}
here $\langle 1,1,p\rangle = e_1 \otimes \sum_1^p e_i \otimes e_i$ is a special instance of the matrix multiplication tensor; note that in fact $\langle 1,1,p\rangle$ is an identity matrix of size $p$, regarded as a tensor on three factors.

\autoref{lem:interpolatewithmatrixanydeg} shows that if $T \trianglerighteq S$ then there exists a $p$ such that $T^{\thickdot p} \geq S$. We call $\langle 1,1,p \rangle$ the \textit{aiding matrix} and $p$ the rank of the aiding matrix. We study upper bounds on the rank of an aiding matrix in terms of approximation and error degree of a degeneration and a partial degeneration. 

In~\autoref{subsec:Interpolation with a matrix}, we show that for partial degenerations of approximation degree $d$ and error degree $e$, the rank of the aiding matrix can be equal to $\min\{ d+1,e+1\}$. More precisely, if $T \bgeq^e_d S$ then $T^{\thickdot d + 1} \geq S$ and $T^{\thickdot e + 1} \geq S$. Even more strikingly we show that if $T \in U_1 \otimes U_2 \otimes U_3$, where the space $U_1 \otimes U_2 \otimes U_3$ is prehomogeneous under the action of $\text{GL}(U_2) \times \text{GL}(U_3)$, and if the orbit of $T$ is dense, then $T^{\thickdot 2} \geq S$ for any other tensor $S \in U_1 \otimes U_2 \otimes U_3$. Using methods from algebraic geometry, we generalize this observation to the setup where the orbit closure is a lower dimensional variety.

It turns out that these findings are in strong contrast to the case of degenerations that are not partial degenerations. To see this, in~\autoref{subsec:A substitution method for aided rank}, we develop a method to give lower bounds on the minimal possible rank of an aiding matrix. This relies on a variant of the \textit{substitution method}~\cite{5959837}.  In~\autoref{subsec:A substitution method for aided rank}, we define the notion of \textit{aided rank} as 
\begin{equation*}
    R^{\thickdot p} (T) = \text{min}\lbrace r : \langle r \rangle ^{\thickdot p} \geq T \rbrace.
\end{equation*}
When $p = 1$, this reduces to the notion of tensor rank \cite[Prop. 5.1.2.1]{Lan:GeometryComplThBook}. We show that one can generalize the substitution method to give lower bounds on aided rank and on the minimal possible rank of an aiding matrix for several examples of degeneration. For example, for the degeneration
\begin{equation}\label{eq:introaidedrank}
    \langle 2^k \rangle \trianglerighteq W^{\boxtimes k}
\end{equation}
we show that $R^{\thickdot 2^{k}-1}(W^{\boxtimes k}) \geq  2^k + 1$. In other words, the minimal rank $p$ of an aiding matrix turning the degeneration in~\autoref{eq:introaidedrank} into a restriction is $2^k$. Note that for this example, the no-go result for partial degenerations~\autoref{prop:nogoresult} gives that the degeneration cannot be realized as a partial degeneration. Also note that the minimal possible rank of an aiding matrix differs from our naive upper bound from~\autoref{lem:interpolatewithmatrixanydeg} only by a factor of $1/2$. 

\subsection{Conclusion and outlook}
In this work, we introduce and study for the first time partial degenerations which is the natural intermediate notion of restriction and degeneration of tensors. We believe that studying partial degenerations can yield deeper insight in the theory of restriction and degeneration of tensors and thereby also deeper insight into combinatorial and complexity theoretic questions as well as physical processes. 

In fact, part of the motivation to introduce aided rank and partial degeneration was the vision to use them as a tool to study the conversion of entanglement structures, in the spirit of \cite{tensornetworkfromgeometry}. Indeed, our results on aided rank find application in this context, see \cite[Section IV.A]{CLSWW:ResourceTheory}. On the lattice, tensors are associated to plaquettes touching only on adjacent vertices, rather than on all vertices as with the Kronecker product. Additional EPR pairs on the side of the plaquettes might therefore be moved from one plaquette to the other, enabling unexpected lattice transformations.

Furthermore, interpolation has proven a useful tool in lattice conversion \cite{tensornetworkfromgeometry,ChrGesStWer:Optimization}, but requires a unit tensor over the entire lattice. If degeneration is only to take place on a subset of the vertices a less demanding more localized entangled state may be sufficient; an idea that awaits its full exploration in future work. 

In fact, we emphasize that partial degenerations allow for a more economical conversion into restrictions than honest degenerations. In the three-particle SLOCC entanglement context, this means that instead of requiring a global GHZ state, we only need an EPR pair combining the vertices on which the non-constant maps are applied thus requiring fewer resources. The degree of the degeneration is additive under taking tensor copies, independently on whether this happens as a Kronecker product, or with more general structures, see \cite{ChrGesJen:BorderRankNonMult}. On the other hand, the interpolation argument is only linear in the degree. Therefore for $n$-copy SLOCC operations only $O(\log n)$ 2-level EPR pairs are required. Again, the full potential of this proposition remains to be explored in future work. 

We identify some open problems arising from this work:
\begin{enumerate}
    \item Consider the tensor 
    \begin{equation*}
        \lambda = e_1 \wedge e_2 \wedge e_3 + e_3 \otimes e_3 \otimes e_3 \in \bbC^3 \otimes \bbC^3 \otimes \bbC^3
    \end{equation*}
    where $\wedge$ is the antisymmetric product. In~\cite{tensornetworkfromgeometry}, the fact that $\langle 2,2,2 \rangle \trianglerighteq \lambda$ was used to construct an efficient representation in the `projected entangled pair' (PEPS) formalism of the `resonating valence bond state' (RVB state)~\cite{ANDERSON1973153}. It is known that $\langle 2,2,2 \rangle \not\geq \lambda$ and we conjecture that a partial degeneration is not possible either. 
    \item Consider a degeneration
    \begin{equation*}
        (A_1(\epsilon) \otimes A_2 (\epsilon) \otimes A_3 (\epsilon))T = \epsilon^d S + \sum_{i = 1}^e \epsilon^{d + i} S_i.
    \end{equation*}
    The approximation degree $d$ and the error degree $e$ control the sizes of tensors in interpolation results which are used to transform degenerations into restrictions, see~\cite{BiCaLoRo:O277ComplexityApproximateMatMult,Christandl2017_TR}. These degrees are not well understood. We propose that it might be more tangible to study the approximation and error degrees in the context of partial degenerations.
    \item The fact that degenerations can be turned into restrictions by interpolating with a unit tensor (see~\autoref{thm:interpolatingasusual}) is the key ingredient that allows one to use \emph{border rank} in the study of matrix multiplication complexity, see e.g. \cite{BiCaLoRo:O277ComplexityApproximateMatMult,10.1145/28395.28396}. It has been used more recently to construct more efficient tensor network representations of quantum states~\cite{tensornetworkfromgeometry} and for many other theoretical applications like showing that tensor rank and tensor border rank are not multiplicative under the tensor product~\cite{Christandl2017_TR,ChrGesJen:BorderRankNonMult}. It is interesting to ask for similar applications when interpolating with a matrix instead of a unit tensor as in~\autoref{lem:interpolatewithmatrixanydeg}. 
    \item In this work, we focused on tensors of order three over complex numbers. Many of our results generalize to the more general setup of $k$-fold tensor products of vector spaces over any field and it would be interesting to further explore these situations. For higher order tensor products, one potentially has a hierarchy of partial degenerations, depending on the number of linear maps depending on $\eps$. 
    \item In the recent work~\cite{Ben}, the authors prove uniqueness results and polynomial time algorithms for $r$-aided rank decompositions. They generalize their results to so-called $(\mathcal{X},\mathcal{W})$-decompositions. Many of our results on $r$-aided rank also can be generalized to this setup and it would be interesting to further explore this direction.
\end{enumerate}

\subsection{Acknowledgements}
M.C., V.L. and V.S. acknowledge financial support from the European Research Council (ERC Grant Agreement No. 81876), VILLUM FONDEN via the QMATH Centre of Excellence (Grant No.10059) and the Novo Nordisk Foundation (Grant No. NNF20OC0059939 ‘Quantum for Life’). F.G.'s work is partially supported by the Thematic Research Programme ``Tensors: geometry, complexity and quantum entanglement'', University of Warsaw, Excellence Initiative -- Research University and the Simons Foundation Award No. 663281 granted to the Institute of Mathematics of the Polish Academy of Sciences for the years 2021--2023.

\section{Preliminaries}\label{sec:Preliminaries}

In this section, we revisit the theory of restrictions and degenerations of tensors. We introduce a few special tensors that will be used in the paper and mention their relation to algebraic complexity theory and quantum information theory. For a thorough introduction to geometric aspects of degeneration of tensors we refer to~\cite{Lan:TensorBook}. More details about the relation to algebraic complexity theory can be found in~\cite{1997-buergisser,gs005}. An introduction to tensor networks can be found in~\cite{RevModPhys.93.045003}.

Let $U_1, U_2 , U_3$ and $V_1, V_2 ,  V_3$ be complex vector spaces of dimensions $u_1,u_2,u_3$ and $v_1,v_2,v_3$, respectively, and consider tensors $T \in U_1 \otimes U_2 \otimes  U_3$ and $S \in V_1\otimes  V_2 \otimes  V_3$. We say that $T$ \emph{restricts} to $S$ and write $T \geq S$ if there are linear maps $A_i : U_i \rightarrow V_i$ such that $S = (A_1 \otimes A_2 \otimes A_3)T$. We say that $T$ \emph{degenerates} to $S$ and write $T \trianglerighteq S$ if there are linear maps $A_i(\epsilon) : U_i \rightarrow V_i$ depending polynomially on $\epsilon$ such that 
\begin{equation*}
  (A_1(\epsilon) \otimes A_2(\epsilon) \otimes A_3(\epsilon))T =  \epsilon^d S + \epsilon^{d+1}S_1 + \dots  + \epsilon^{d+e}S_e  
\end{equation*}
for some natural numbers $d,e$ and some tensors $S_1, \dots , S_e$. The quantities $d$ and $e$ are called approximation degree and error degree, respectively. Sometimes, we want to keep track of $d$ and $e$ and write $T \trianglerighteq_{d}^e S$. 

Recall the \emph{Kronecker product} of tensors. Fix bases in the tensor factors and let $T \in U_1 \otimes U_2 \otimes U_3$ and $S \in V_1 \otimes V_2 \otimes V_3$ be specified in these bases by coefficients $T_{i_1 ,i_2, i_3}$ and $S_{j_1, j_2, j_3}$, that is, 
\begin{equation}\label{eq:TandSincoordinates}
    T = \sum_{i_1,i_2,i_3 = 1}^{u_1,u_2,u_3} T_{i_1,i_2,i_3} e_{i_1} \otimes e_{i_2} \otimes e_{i_3}, \; S = \sum_{j_1,j_2,j_3 = 1}^{v_1,v_2,v_3} S_{j_1,j_2,j_3} e_{j_1} \otimes e_{j_2} \otimes e_{j_3}.
\end{equation}
We define their Kronecker product $T \boxtimes S \in (U_1 \otimes V_1) \otimes (U_2 \otimes V_2) \otimes (U_3 \otimes V_3)$ as 
\begin{equation*}
    T\boxtimes S = \sum_{i_1,i_2,i_3 ,j_1,j_2,j_3  } T_{i_1,i_2,i_3}\cdot S_{j_1,j_2,j_3} (e_{i_1}\otimes e_{j_1}) \otimes (e_{i_2}\otimes e_{j_2}) \otimes (e_{i_3}\otimes e_{j_3}).
\end{equation*}
Write $T^{\boxtimes k}$ for $T\boxtimes \dots \boxtimes T$ where the Kronecker product was taken $k$ times. 

Another way of combining two tensors $T$ and $S$ in~\autoref{eq:TandSincoordinates} into one is the \textit{direct sum}. For that, pick a basis $e_{1}, \dots , e_{u_i + v_i}$ of $U_i \oplus V_i$ for $i =1, 2, 3$ and define the tensor $T \oplus S \in (U_1 \oplus V_1) \otimes (U_2 \oplus V_2) \otimes (U_3 \oplus V_3)$ via
\begin{equation*}
    (T \oplus S)_{i_1 , i_2 , i_3} = \begin{cases}
        T_{i_1, i_2, i_3} &\text{ if } i_1 \leq u_1, i_2 \leq u_2, i_3 \leq u_3, \\
        S_{i_1-u_1, i_2-u_2, i_3-u_3} &\text{ if } i_1 > u_1, i_2 > u_2, i_3 > u_3, \\
        0 &\text{ otherwise.}
    \end{cases}
\end{equation*}
We visualize the direct sum of two tensors $T$ and $S$ in~\autoref{fig:directsum}.

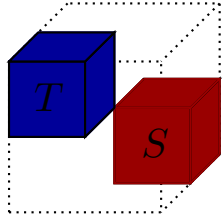
\begin{figure}
    \centering
    \begin{tikzpicture}
\draw[thick,black,dotted]  (2,0,0) -- (2,2,0) -- (0,2,0);
\draw[thick,black,dotted] (2,0,0) -- (2,0,2);
\draw[thick,black,dotted] (2,2,0) -- (2,2,2);
\draw[thick,black,dotted] (0,2,0) -- (0,2,2);
\draw[thick,black,dotted] (0,0,2) -- (2,0,2) -- (2,2,2) -- (0,2,2) -- cycle;
\draw[color = white, fill = white] (1,0,-.4) circle (.2);

\begin{scope}[shift = {(0,-1,0)}]
\begin{scope}[shift = {(1,1,0)}]
    \draw[thick,fill = red!60!black,opacity = .1] (0,0,1) -- (1,0,1) -- (1,1,1) -- (0,1,1) -- cycle;
\draw[thick,fill = red!60!black,opacity = .3] (1,0,0) -- (1,0,1) -- (1,1,1) -- (1,1,0) -- cycle;
\draw[thick,fill = red!60!black,opacity = .1] (0,1,0) -- (0,1,1) -- (1,1,1) -- (1,1,0) -- cycle;
\end{scope}
    \draw[thick,red!60!black] (2,1,0) -- (2,2,0) -- (1,2,0);
\draw[thick,red!60!black] (2,1,0) -- (2,1,1);
\draw[thick,red!60!black] (2,2,0) -- (2,2,1);
\draw[thick,red!60!black] (1,2,0) -- (1,2,1);
\draw[thick,red!60!black] (1,1,1) -- (2,1,1) -- (2,2,1) -- (1,2,1) -- cycle;
\end{scope}
\begin{scope}[shift = {(0,1,1)}]
    \draw[thick,blue!60!black] (1,0,0) -- (1,1,0) -- (0,1,0) ;
\draw[thick,blue!60!black] (1,0,0) -- (1,0,1);
\draw[thick,blue!60!blue] (1,1,0) -- (1,1,1);
\draw[thick,blue!60!black] (0,1,0) -- (0,1,1);
\draw[thick,blue!60!black] (0,0,1) -- (1,0,1) -- (1,1,1) -- (0,1,1) -- cycle;
\draw[thick,fill = blue!60!black,opacity = .1] (0,0,1) -- (1,0,1) -- (1,1,1) -- (0,1,1) -- cycle;
\draw[thick,fill = blue!60!black,opacity = .3] (1,0,0) -- (1,0,1) -- (1,1,1) -- (1,1,0) -- cycle;
\draw[thick,fill = blue!60!black,opacity = .1] (0,1,0) -- (0,1,1) -- (1,1,1) -- (1,1,0) -- cycle;
\end{scope}
\node[scale = 1.4] at (-.4,0.6,-.4){$T$}; 
\node[scale = 1.4] at (1,0,-.4){$S$}; 
\end{tikzpicture}
\caption{Visualizing the direct sum of two tensors $T$ and $S$: The tensor $T \oplus S$ is block-diagonal where one block is the tensor $T$ and the other block the tensor $S$.}
    \label{fig:directsum}
\end{figure}

A third order tensor $T \in U_1 \otimes U_2 \otimes U_3$ defines naturally three linear maps $U_1^* \to U_2 \otimes U_3$, $U_2^* \to U_1 \otimes U_3$, $U_3^* \to U_1 \otimes U_2$, called the flattening maps of $T$. We say that $T$ is concise if the three flattenings are injective. Equivalently, $T$ is concise if there are no subspace $U'_i \subseteq U_i$, with at least one strict inclusion, such that $T \in U'_1 \otimes U'_2 \otimes U'_3$. The image $T(U_1^*)$ is a linear subspaces of $ U_2 \otimes U_3 = \Hom(U_2^*,U_3)$, which can naturally be identified with a linear space of matrices. 
It is an immediate fact that the linear space $T(U_1^*)$ uniquely determines $T$ up to the action of $\GL(U_1)$. We often identify a linear space of matrices with a tensor defining it. This point of view is classical, but it turned out to be extremely useful in recent work in the study of tensor restriction and degeneration \cite{homs2022bounds,JelLandPal:ConciseTensorsMinimalBR,ChrGesZui:GapSubrank}.

We will now introduce a few special tensors that will be important throughout this work. Let $U$ be an $r$-dimensional vector space with basis $e_1 ,\dots ,e_r$. We call $$\langle r \rangle = e_1 \otimes e_1 \otimes e_1 + \dots + e_r \otimes  e_r \otimes e_r$$ the $r$-th \emph{unit tensor}.  Restriction and degeneration of the unit tensor define the notions of \emph{rank} and \emph{border rank} of tensors: given $S \in V_1 \otimes V_2 \otimes V_3$, the rank of $S$ is $R(S) = \min \lbrace r : \langle r \rangle \geq S \rbrace$; the border rank of $S$ is $\underline{R}(S) = \min \lbrace r: \langle r \rangle \trianglerighteq S \rbrace$.  

Another important tensor is $W = e_1 \otimes  e_1 \otimes e_2 + e_1 \otimes  e_2 \otimes e_1 + e_2 \otimes e_1 \otimes e_1$. It is the smallest possible example for an honest degeneration. In fact, while one can show that $R(W) = 3$, one easily sees $\langle 2 \rangle \trianglerighteq W$ via
\begin{equation*}
	\epsilon W = (e_1 + \epsilon e_2)^{\otimes 3} - e_1^{\otimes 3} + \mathcal{O}(\epsilon^2), 
\end{equation*}
hence $\underline{R}(W)=2$. In~\autoref{prop:nogoresult}, we will see that in fact $\langle 2 \rangle$ does not partially degenerate to $W$.  

A tensor that will be important is the \emph{matrix multiplication tensor} defined for any $m,n,p$ as 
\begin{equation*}
	\langle m,n,p \rangle = \sum_{i,j,k = 1}^{m,n,p} (e_i 
\otimes e_j) \otimes (e_j \otimes e_k) \otimes (e_k \otimes e_i) \in (\mathbb{C}^m \otimes \mathbb{C}^n) \otimes (\mathbb{C}^n \otimes \mathbb{C}^p) \otimes (\mathbb{C}^p \otimes \mathbb{C}^m).
\end{equation*}
We will in particular use matrix multiplication tensors of the form $\langle 1,1,p \rangle$. For a tensor $T \in U_1 \otimes U_2 \otimes U_3$, we will often consider the tensor $T^{\thickdot p} = T \boxtimes \langle 1,1,p \rangle$.

The tensors mentioned in this section have natural applications in algebraic complexity theory. A tensor $T \in U_1 \otimes  U_2 \otimes U_3$ naturally defines a bilinear map $T: U_1^* \times U_2^* \rightarrow U_3$. For example, the bilinear map induced by $\langle r \rangle$ multiplies two $r$-dimensional vectors entry-wise. The induced bilinear map for the matrix multiplication tensor $\langle m,n,p \rangle$ is the matrix multiplication, mapping two matrices, of size $m \times n$ and $n \times p$ respectively, to (the transpose of) their product, which is a matrix of size $m \times p$. In this context, the rank of a tensor encodes the number of scalar multiplications needed to evaluate the associated bilinear map. For example, a major open problem in algebraic complexity theory is to determine the \emph{exponent of matrix multiplication} $\omega$. This is the minimal $\omega$ such that for any $\epsilon >0$ one can multiply $n \times  n$ matrices using $\mathcal{O}(n^{\omega + \epsilon})$ multiplications; it can be defined in terms of rank and border rank as the minimum $\omega$ such that $R(\langle n,n,n\rangle)$ or equivalently $\underline{R}(\langle n,n,n \rangle)$ is in $\calO(n^\omega)$ \cite{1997-buergisser,gs005}.

Moreover, the tensors introduced so far have a natural interpretation in quantum information theory as well. In quantum physics, tensors correspond to multiparticle quantum states and the notion of restriction is known as conversion under \emph{stochastic local operations and classical communication (SLOCC)}. In this context, the unit tensor is known as the unnormalized Greenberger-Horn-Zeilinger (GHZ) state. The $W$ tensor also plays an important role: The fact that $\langle 2 \rangle\ngeq W$  and $W \ngeq \langle 2 \rangle$ was used in~\cite{threequbitsthreedifferentways} to observe that three qubits can be genuinely three-party entangled in two inequivalent ways. The tensor $\langle 1,1,p \rangle$ has a natural interpretation as well: It describes a quantum state on three parties where the second and third parties share an \textit{Einstein-Podolsky-Rosen pair} (EPR pair) on $p$ levels. For any tensor $T$, the tensor $T^{\thickdot p}$ is the overall (unnormalized) quantum state when all three parties share a quantum state associated to $T$ and the second and third party in addition are given an EPR pair on $p$ levels. 

\section{Partial degeneration}\label{sec:Degenerations with one constant map}

In this section, we introduce and study the notion of partial degeneration, a natural intermediate notion between restriction and degeneration. 
In~\autoref{sec:defandmot}, we will define partial degeneration. After that, we review in~\autoref{Strassenexample} a known example of a partial degeneration. In~\autoref{subsec:Nullsubtensors}, we will recall a property of tensors called \textit{compressibility} and demonstrate with an example how this can be used to rule out restriction. We will see more examples in~\autoref{subsec:Matrix pencils} using the theory of prehomogeneous tensor spaces. Finally, in~\autoref{subsec:A no-go result for partial degenerations} we will study situations where no honest partial degeneration exist.

\subsection{Definition and motivation}\label{sec:defandmot}

The main concept of this section is the following special version of degeneration, intermediate between restriction and fully general degeneration.

\begin{definition}\label{def:partialdeg}
Let $T \in U_1 \otimes U_2 \otimes U_3$ and $S \in V_1 \otimes V_2 \otimes V_3$ be tensors. We say that $T$ \textit{degenerates partially} to $S$ and write $T \bgeq S$ if there is a linear map $A_1 \colon U_1 \rightarrow V_1$ and linear maps $A_2(\epsilon), A_3(\epsilon)$ with entries in the polynomial ring $\mathbb{C}[\epsilon ]$ such that 
\begin{equation*}
    (A_1 \otimes A_2 (\epsilon) \otimes A_3 (\epsilon ))T = \epsilon^d S + \epsilon^{d+1} S_1 + \dots + \epsilon^{d + e} S_e.
\end{equation*}
We sometimes write $T \bgeq_d^e S$ to keep track of $d$ and $e$. We call a partial degeneration $T \bgeq S$ an \textit{honest partial degeneration} if $T$ does not restrict to $S$.
\end{definition}

It is clear that every restriction is a partial degeneration and every partial degeneration is a degeneration. This raises the following two questions:
\begin{enumerate}[(i)]
    \item Can every partial degeneration $T \bgeq S$ be realized as a restriction $T \geq S$? 
    \item Can every degeneration $T \trianglerighteq S$ be realized as a partial degeneration $T \bgeq S$?
\end{enumerate}
We point out that only allowing one of the three linear maps to depend on $\epsilon$ provides the same notion as restriction:
\begin{remark}\label{rmk:one moving map}
Let $T \in U_1 \otimes U_2 \otimes U_3$, $S \in V_1 \otimes V_2 \otimes V_3$ be tensors and suppose there are linear maps $A_1,A_2,A_3(\eps)$ with $A_i : U_i \to V_i$, and $A_3(\eps)$ depending polynomially on $\eps$ such that $S = \lim_{\eps \to 0} \frac{1}{\eps^d} ( A_1 \otimes A_2 \otimes A_3(\eps) T)$. Then $S$ is a restriction of $T$. To see this, write $A_3 (\eps) = \eps^d A_{3,d} + \cdots \eps^{d+e} A_{3,e}$ with $A_{3,j} : U_3 \to V_3$. It is immediate that $S = (A_1 \otimes A_2 \otimes A_{3,d})T$; this expresses $S$ as a restriction of $T$. 
\end{remark}
In~\autoref{Strassenexample}, \autoref{subsec:Nullsubtensors} and \autoref{subsec:Matrix pencils}, we provide families of examples demonstrating that the answer to question in (i) is negative. In~\autoref{subsec:A no-go result for partial degenerations}, we show that the answer to question (ii) is negative as well.

\subsection{Strassen's tensor}\label{Strassenexample}
In~\cite{STRASSEN+1987+406+443}, a first example of a partial degeneration was found: Let $U_1 \simeq \mathbb{C}^{q-1}$ and $U_2 \simeq U_3 \simeq \mathbb{C}^{q}$ and consider the tensor
\begin{equation*}
    \text{Str}_q = \sum_{i = 1}^{q-1} e_{i} \otimes e_i \otimes e_q + e_i \otimes e_q \otimes e_i \in U_1 \otimes U_2 \otimes U_3. 
\end{equation*}
It is not hard to see that $R(\Str_q) = 2q-2$. On the other hand, it is a partial degeneration of the unit tensor $\langle q \rangle$ via 
\begin{equation*}
    \epsilon \Str_q = \sum_{i = 1}^{q - 1} e_i \otimes (e_q + \epsilon e_i ) \otimes (e_q + \epsilon e_i ) - \left( \sum_{i = 1}^{q-1} e_i \right) \otimes e_q \otimes e_q + \mathcal{O}(\epsilon^2 ).
\end{equation*}
In~\autoref{subsec:A no-go result for partial degenerations}, we will show that these are essentially all partial degenerations of $\langle r \rangle$ that can be found in $U_1 \otimes U_2 \otimes U_3$ with $\dim (U_1) = {r-1}$.

\subsection{Compressibility of tensors and matrix multiplication}
 \label{subsec:Nullsubtensors}

In this section, we will find a family of examples of partial degeneration of the $2 \times 2$-matrix multiplication tensor. One challenge of finding an honest partial degeneration is to show that this partial degeneration is actually not a restriction. To achieve that, we recall the notion of \textit{compressibility}~\cite{DBLP:journals/corr/LandsbergM16a}.

\begin{definition}
A tensor $T \in U_1 \otimes U_2 \otimes U_3$ is $(a_1, a_2, a_3)$\it{-compressible} if there are linear maps $A_i: U_i \rightarrow U_i$ of rank $a_i$ such that $(A_1 \otimes A_2 \otimes A_3) T = 0$.
\end{definition}

\begin{figure}
    \centering
   \begin{tikzpicture}
\draw[thick,black,dotted]  (2,0,0) -- (2,2,0) -- (0,2,0);
\draw[thick,black,dotted] (2,0,0) -- (2,0,2);
\draw[thick,black,dotted] (2,2,0) -- (2,2,2);
\draw[thick,black,dotted] (0,2,0) -- (0,2,2);
\draw[thick,black,dotted] (0,0,2) -- (2,0,2) -- (2,2,2) -- (0,2,2) -- cycle;
\draw[color = white, fill = white] (1,0,-.4) circle (.2);
\begin{scope}[shift = {(0,-1,0)}]
\begin{scope}[shift = {(1,1,0)}]
    \draw[thick,fill = red!60!black,opacity = .1] (0,0,1) -- (1,0,1) -- (1,1,1) -- (0,1,1) -- cycle;
\draw[thick,fill = red!60!black,opacity = .3] (1,0,0) -- (1,0,1) -- (1,1,1) -- (1,1,0) -- cycle;
\draw[thick,fill = red!60!black,opacity = .1] (0,1,0) -- (0,1,1) -- (1,1,1) -- (1,1,0) -- cycle;
\end{scope}
    \draw[thick,red!60!black] (2,1,0) -- (2,2,0) -- (1,2,0);
\draw[thick,red!60!black] (2,1,0) -- (2,1,1);
\draw[thick,red!60!black] (2,2,0) -- (2,2,1);
\draw[thick,red!60!black] (1,2,0) -- (1,2,1);
\draw[thick,red!60!black] (1,1,1) -- (2,1,1) -- (2,2,1) -- (1,2,1) -- cycle;
\end{scope}
\node[scale = 1.4] at (1,0,-.4){$0$}; 
\end{tikzpicture}
    \caption{Visualization of an $(a_1, a_2, a_3)$-compressible tensor: This large $u_1 \times u_2 \times u_3$-cube is the tensor $T \in U_1 \otimes U_2 \otimes U_3$. The entries of $T$ -- specified in some fixed basis -- can be written in the cells of this cube. The smaller, red $a_1 \times a_2 \times a_3$-cube depicts a block of size $a_1\times a_2\times a_3$ where each entry of $T$ equals zero. By choosing the linear maps as projectors onto the last $u_1 - a_1$ resp.~$u_2 - a_2$ resp.~$u_3 - a_3$ coordinates, we see that each such tensor is $(a_1,a_2,a_3)$-compressible.}
    \label{fig:nullsubtensor}
\end{figure}
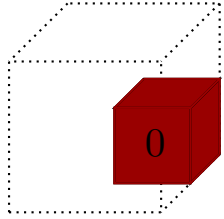

Equivalently, $T$ is $(a_1, a_2, a_3)$-compressible if there are linear subspaces $U_i' \subseteq U_i^*$ with $\dim U_i ' = a_i$, such that, as a trilinear map $T|_{U_1' \times U_2' \times U_3'} \equiv 0$. In coordinates, this is equivalent to the existence of bases of the spaces $U_1, U_2$ and $U_3$ such that, in these bases, $T_{i_1, i_2, i_3} =0$ if $i_j \geq \dim U_j - a_j$. We visualize the concept of an $(a_1, a_2, a_3)$-compressible tensor in~\autoref{fig:nullsubtensor}. The following technical result will become handy later.

\begin{lemma}\label{lem:nullsubtensor}
Let $T \in U_1 \otimes U_2 \otimes U_3$ and $S \in V_1 \otimes V_2 \otimes V_3$. Let $T \geq S$ and let $S$ be concise. If $S$ is $(a_1, a_2, a_3)$-compressible then $T$ is $(a_1, a_2, a_3)$-compressible.
\end{lemma}
\begin{proof}
By assumption, there are maps $A_i$ with rank $a_i$ such that $(A_1 \otimes A_2 \otimes A_3) S = 0$. As $S$ is concise, the restriction maps $M_i$ must be surjective where $S = (M_1 \otimes M_2 \otimes M_3) T$. Therefore, the maps $A_1M_1, A_2M_2$ and $A_3M_3$ also have rank $a_1,a_2$ and $a_3$, respectively. Since $(A_1M_1 \otimes A_2M_2 \otimes A_3M_3) T =  (A_1 \otimes A_2 \otimes A_3)S= 0$ the claim follows.
\end{proof}

\autoref{lem:nullsubtensor} can be used to exclude restrictions $T\geq S$ if $T$ is less compressible than $S$. An example of a tensor that is not ``very compressible'' is the matrix multiplication tensor.

\begin{lemma}\label{lem:mamuhasnonullsubtensor}
The $2\times 2$ matrix multiplication tensor $\langle 2,2,2 \rangle$ is not $(2,3,3)$-compressible.
\end{lemma}
\begin{proof}
Note that any $4 \times 4$ matrix $M = (M_{(u,v),(x,y)})_{u,v,x,y = 1,2}$ (labelled by double indices) induces a linear endomorphism of the space of $2 \times  2$ matrices via 
\begin{align*}
	M: \bbC^2 \otimes \bbC^2 &\to \bbC^2 \otimes \bbC^2 \\
 x &\mapsto M.x,\; \text{ where } (M.x)_{i,j} = \sum_{k,l = 1,2} M_{(i,j),(k,l)}x_{k,l}.
\end{align*}

Recall that $\langle 2,2,2 \rangle$ corresponds to calculating the four bilinear forms $z_{j,i} = x_{i1}y_{1j} + x_{i2}y_{2j}$ for $i,j = 1,2$, that is, the entries of $(x\cdot y)^T$ where the entries of $x$ and $y$ are regarded as variables.

Let $S = (A_1 \otimes A_2 \otimes A_3)\langle 2,2,2 \rangle$ be a restriction of the $2\times 2$ matrix multiplication tensor. Interpreting $A_1, A_2$ and $A_3$ as $4 \times  4$ matrices, an easy calculation shows that the four bilinear forms corresponding to the tensor $S$ are the four entries of the transpose of  
\begin{equation}\label{eq:nonullsubmamu}
A_3.\left( (A_1.x)\cdot (A_2.y) \right). 
\end{equation}

Now, let the rank of $A_1$ and $A_2$ be at least 3 and the rank of $A_3$ be at least 2. It is clear that the space of all $A_1.x$ for $x \in M_{2 \times  2}$ is at least 3-dimensional (the same holds for $A_2$). It is well-known that every subspace of $M_{2 \times  2}$ of dimension at least $3$ must contain an invertible matrix. Choosing $x_0\in M_{2 \times  2}$ such that $A_1 . x_0$ is invertible, we see that the space of matrices of the form  $(A_1.x_0)\cdot (A_2.y)$ for $y \in M_{2 \times  2}$ contains three linearly independent matrices. Hence, since we assumed that $A_3$ has rank $\geq 2$, we see that~\autoref{eq:nonullsubmamu} cannot be identically $0$. This finishes the proof.
\end{proof}

We record the following technical result, describing restrictions, degenerations and partial degenerations of the matrix multiplication tensor. Analogous statements are true for more general tensor networks, see e.g. \cite{CzMiSe:UniformMPSfromAG}.


\begin{lemma}\label{lem:MaMurestrictions}
Let $V_1,V_2,V_3$ be vector spaces with dimensions $v_1,v_2,v_3$ and let $S \in V_1\otimes V_2\otimes V_3$. 
Then $\langle m,n,p\rangle \geq S$ if and only if there exist three families of matrices
\begin{equation}\label{eq:tensornetworkrep}
\begin{aligned}
	&\alpha_1,\dots , \alpha_{v_1} \in \mathbb{C}^{m \times n} \\
 	&\beta_1 ,\dots ,\beta_{v_2} \in \mathbb{C}^{n \times p}, \\
	&\gamma_1 ,\dots ,\gamma_{v_3} \in \mathbb{C}^{p \times m} 
\end{aligned}
\end{equation}
such that for every $i,j,k$
\[
S_{ijk} = \trace(\alpha_i\beta_j\gamma_k).
\]
Moreover, $\langle m,n,p \rangle \trianglerighteq S$ if and only if there are matrices as in \eqref{eq:tensornetworkrep}, depending on a variable $\eps$, and an integer $d$ such that 
\[
\eps^d S_{ijk} = \trace (\alpha_i(\eps) \beta_j(\eps) \gamma_k(\eps) ) + O(\eps^{d+1}).
\]
In particular, if the matrices $\alpha_i$ can be chosen constant in $\eps$, then $\langle m,n,p\rangle \bgeq S$.
\end{lemma}

We will apply \autoref{lem:MaMurestrictions} to characterize degenerations of the matrix multiplication tensor $\langle 2 ,2,2\rangle$.

\begin{proposition}\label{prop:333constnatmap}
	Let $S \in \mathbb{C}^{3}\otimes \mathbb{C}^{4}\otimes \mathbb{C}^{4}$ be a concise $(3,3,3)$-compressible tensor. Then $S$ is an honest partial degeneration of $\langle 2,2,2 \rangle$. 
\end{proposition}
\begin{proof}
Fixing bases, we can write our tensor $S$ as 
\begin{equation*}
	S = \sum_{i,j,k = 1}^{3,4,4} S_{i,j,k}e_i \otimes e_j \otimes e_k
\end{equation*}
such that $S_{i,j,k} = 0$ whenever both $j$ and $k$ are at least $2$.

From~\autoref{lem:MaMurestrictions}, it suffices to find $2 \times 2 $ matrices 
\begin{equation*}
\alpha_1 ,\dots , \alpha_3 \in \mathbb{C}^{m \times n}, \;
\beta_1(\epsilon) ,\dots ,\beta_4(\epsilon) \in \mathbb{C}[\epsilon]^{n \times p}, \;
\gamma_1(\epsilon) ,\dots ,\gamma_4(\epsilon) \in \mathbb{C}[\epsilon]^{p \times m}
\end{equation*}
such that
\begin{equation}\label{eq:mpsnullsubtensor}
	\epsilon S_{i,j,k} = \trace(\alpha_i \beta_j \gamma_k) + \mathcal{O}(\epsilon^{2}).
\end{equation}

Choosing matrices 

\begin{center}
\begin{tabular}{c  c  c}
	\vspace{1em}
$\alpha_1 = \begin{pmatrix}
      1&0\\
      0&-1\end{pmatrix}$ & $\beta_1 = \begin{pmatrix}
      \epsilon(S_{1,1,1}-1)+1&\,S_{2,1,1}\\
      S_{3,1,1}&1\end{pmatrix}$ &  $\gamma_1 = \begin{pmatrix}
      \epsilon+1&0\\
0&1\end{pmatrix}$\\\vspace{1em}
 $\alpha_2 = \begin{pmatrix}
      0&0\\1&0\end{pmatrix}$ & $\beta_2 = \begin{pmatrix}
      \epsilon S_{1,2,1}&\epsilon S_{2,2,1}\\
\epsilon S_{3,2,1}&0\end{pmatrix}$& $\gamma_2 = \begin{pmatrix}
      \epsilon\,S_{1,1,2}&\epsilon\,S_{2,1,2}\\
      \epsilon\,S_{3,1,2}&0\end{pmatrix}$ \\\vspace{1em}
      $\alpha_3 = \begin{pmatrix}
      0&1\\
0&0\end{pmatrix}$ &  $\beta_3 = \begin{pmatrix}
      \epsilon S_{1,3,1}&\epsilon S_{2,3,1}\\
     \epsilon S_{3,3,1}&0\end{pmatrix}$& $\gamma_3 = \begin{pmatrix}
      \epsilon\,S_{1,1,3}&\epsilon\,S_{2,1,3}\\
\epsilon\,S_{3,1,3}&0\end{pmatrix}$\\\vspace{1em}
&$ \beta_4 = \begin{pmatrix}
      0&\epsilon S_{2,4,1}\\
      \epsilon S_{3,4,1}&-\epsilon S_{1,4,1}\end{pmatrix}$
&$\gamma_4 = \begin{pmatrix}
      0&\epsilon\,S_{2,1,4}\\
      \epsilon\,S_{3,1,4}&-\epsilon\,S_{1,1,4}\end{pmatrix}$
\end{tabular}
\end{center}
one easily verifies that~\autoref{eq:mpsnullsubtensor} is fulfilled. 

Since $S$ is concise and is $(3,3,3)$-compressible, by~\autoref{lem:nullsubtensor} and \autoref{lem:mamuhasnonullsubtensor} we deduce that $S$ is not a restriction of $\langle 2,2,2 \rangle$, therefore $S$ is an honest partial degeneration of $\langle 2,2,2 \rangle$. 
\end{proof}

An explicit example of a concise and $(3,3,3)$-compressible tensor in $\mathbb{C}^3 \otimes \mathbb{C}^4 \otimes \mathbb{C}^4$ is $\Str_4$ from~\autoref{Strassenexample}. Hence,~\autoref{prop:333constnatmap} proves that $\langle 2,2,2 \rangle \not \geq \Str_4$ and $\langle 2,2,2 \rangle \bgeq \Str_4$.

\begin{remark}\label{rem:orbitdimof222}
\autoref{lem:mamuhasnonullsubtensor} implies that no concise tensor in $\bbC^3 \otimes \bbC^4 \otimes \bbC^4$ which is $(2,3,3)$-compressible is a restriction of $\langle 2,2,2 \rangle$. On the other hand, \autoref{prop:333constnatmap} shows that every $(3,3,3)$-compressible tensor is a partial degeneration of $\langle 2,2,2\rangle$. One might wonder if $(2,3,3)$-compressible tensors are degenerations or partial degenerations of $\langle 2,2,2\rangle$. This is not the case. In fact, the variety of all degenerations of $\langle 2,2,2 \rangle$ in $\mathbb{C}^3\otimes \mathbb{C}^4 \otimes \mathbb{C}^4$ has dimension 31~\cite[Section 5]{https://doi.org/10.48550/arxiv.2101.03148}. However, a simple calculation -- the code for which can be found in~\autoref{codeappendix} -- shows that the orbit closure of a generic element of $\mathbb{C}^3\otimes \mathbb{C}^4 \otimes \mathbb{C}^4$ which is $(2,3,3)$-compressible has dimension $37$; in particular the variety of all $(2,3,3)$-compressible tensors has dimension at least $37$.
\end{remark}

\subsection{Prehomogeneous spaces}\label{subsec:Matrix pencils}

In this section, we will see more examples of partial degenerations by making a connection to the well-studied theory of prehomogeneous tensor spaces.

\begin{definition}
Let $G$ be a group acting on a vector space $V$. We say that $V$ is \emph{prehomogeneous} under the action of $G$ if there is an element $T\in V$ such that $G.T$ is dense in $V$ with respect to the Zariski topology, i.e. $\overline{G.T} = V$.
\end{definition}

Consider the space $U_1 \otimes U_2 \otimes U_3$ where the $U_i$ are vector spaces of dimension $u_i$. Clearly, if $T,S \in U_1 \otimes U_2 \otimes U_3$ such that $S$ is in the orbit closure of $T$ under the action of $\operatorname{GL}(U_2) \times \operatorname{GL}(U_3)$, then $T \bgeq S$. Hence, if $U_1 \otimes U_2 \otimes U_3$ is prehomogeneous under the action of  $\operatorname{GL}(U_2) \times \operatorname{GL}(U_3)$, and $T$ is an element of the dense orbit then all tensors $S \in U_1 \otimes U_2 \otimes U_3$ are partial degenerations of $T$. 

Prehomogeneity of $U_1 \otimes U_2 \otimes U_3$  under the action of $\operatorname{GL}(U_2) \times \operatorname{GL}(U_3)$ only depends on the dimensions of the vector spaces involved and is easy to check~\cite{SatKim:ClassificationIrredPrehomVS}. 

\begin{theorem}\label{thm:prehomrange}
Assume $u_2 \leq u_3$. Define $\lambda(u_1) = \frac{u_1 + \sqrt{u_1^2 - 4}}{2}$.
Then, the space  $U_1 \otimes U_2 \otimes U_3$ is prehomogeneous under the action of $\operatorname{GL}(U_2) \times \operatorname{GL}(U_3)$ if and only if $u_3 > \lambda(u_1) u_2$.
\end{theorem}

Hence, for any choices of $u_1,u_2,u_3$ satisfying the conditions in~\autoref{thm:prehomrange}, there is an element $T \in U_1 \otimes U_2 \otimes U_3$ such that for all $S \in  U_1 \otimes U_2 \otimes U_3$ it holds $T\bgeq S$. 

To show that there exists $S$ which is not a restriction of $T$, we use the following well-known statement. For completeness, we will include a proof.
\begin{lemma}\label{lem:concisesameorbit}
Let $T,S \in U_1 \otimes U_2 \otimes U_3$ be tensors. Assume $S$ is concise.
Then $T\geq S$ if and only if $T$ and $S$ lie in the same $\operatorname{GL}(U_1) \times \operatorname{GL}(U_2) \times \operatorname{GL}(U_3)$-orbit.
\end{lemma}
\begin{proof}
By definition, $T \geq S$ holds if and only if there are linear maps $A_1, A_2$ and $A_3$ such that $(A_1 \otimes A_2 \otimes A_3)T = S$. In particular $S \in (\Im A_1) \otimes (\Im A_2) \otimes (\Im A_3)$. If one of the maps $A_i$ is not invertible, the corresponding image $\Im A_i$ is a proper subspace $U_i$, showing that $S$ is not concise.
\end{proof}

\begin{theorem}\label{thm:prehompartialdeg}
Let $U_1, U_2, U_3$ have dimensions $u_1, u_2, u_3$ such that $\lambda(u_1) u_2 < u_3 < u_1 u_2$, where $\lambda(u_1) = \frac{u_1 + \sqrt{u_1^2 - 4}}{2}$.
Then there exist tensors $T, S \in U_1 \otimes U_2 \otimes U_3$ such that $T \bgeq S$, but $T \not \geq S$.
\end{theorem}
\begin{proof}
We know from~\autoref{thm:prehomrange} that the space $U_1 \otimes U_2 \otimes U_3$ is prehomogeneous under $\operatorname{GL}(U_2) \times \operatorname{GL}(U_3)$.
Let $T$ be a tensor in the dense $\operatorname{GL}(U_2) \times \operatorname{GL}(U_3)$-orbit, so that $T \bgeq S$ for every $S \in U_1 \otimes U_2 \otimes U_3$.

Let $p = u_1 u_2 - u_3$. Note that $u_1 - 1 \leq \lambda(u_1) < u_1$, so $\lambda(u_1) u_2 < u_3 < u_1 u_2$ implies that $0 < p < u_2$.
Define the tensor $S \in U_1 \otimes U_2 \otimes U_3$ as
\[
S = \sum_{i = 1}^{u_1 - 1} e_i \otimes \left(\sum_{j = 1}^{u_2 } e_j \otimes e_{(i-1) u_2 + j}\right) + e_{u_1 } \otimes \left(\sum_{j = 1}^{u_2 - p} e_j \otimes e_{(u_1 - 1) u_2 + j}\right)
\]
It is not hard to see that the tensor $S$ is concise.

To show that $T$ and $S$ lie in different $\operatorname{GL}(U_1) \times \operatorname{GL}(U_2) \times \operatorname{GL}(U_3)$-orbits, we compute the dimensions of these orbits.
Denote $G = \operatorname{GL}(U_1) \times \operatorname{GL}(U_2) \times \operatorname{GL}(U_3)$.

For $T$ we have
$U_1 \otimes U_2 \otimes U_3 \supset \overline{G \cdot T} \supset \overline{[\operatorname{GL}(U_2) \times \operatorname{GL}(U_3)]\cdot T} = U_1 \otimes U_2 \otimes U_3$,
hence $\overline{G\cdot T} = U_1 \otimes U_2 \otimes U_3$ and $\dim G \cdot T = u_1 u_2 u_3$.

For $S$, the dimension of the orbit $G \cdot S$ can be found as $\dim G \cdot S = \dim G - \dim \operatorname{Stab}_G(S)$.
The stabilizer $\operatorname{Stab}_G(S)$ is isomorphic to $P(1, u_1) \times P(u_2 - p, u_2)$ where $P(a, b) \subset \operatorname{GL}_b$ is the parabolic group preserving a subspace of dimension $a$.
Indeed, let $S_i \in U_2 \otimes U_3$ be the slices of $S$ corresponding to the standard basis, that is, $S = \sum_{i = 1}^{u_1} e_i \otimes S_i$.
Note that $\operatorname{rk}(S_i) = u_2$ for $i < u_1 $ and $\operatorname{rk}(S_{u_1 }) = u_2 - p$. Moreover, a nonzero linear combination $\sum_{i = 1}^{u_1} \alpha_i S_i$ has rank $u_1 - p$ if and only if $\alpha_i = 0$ for $i \leq u_2 - 1$.
It follows that if $(A \otimes B \otimes C) S = S$, then $A$ preserves the $1$-dimensional subspace $\left<e_{u_1}\right>$.
Therefore, we have $a_{u_1, u_1} (B \otimes C) S_{u_1} = S_{u_1}$ and it follows that $B$ preserves the $(u_2 - p)$-dimensional subspace $\left<e_{1},\dots,e_{u_2  - p}\right>$, which is the image of $S_{u_1 }$ considered as a linear map $U_3^* \to U_2$.
Now, given $A$ and $B$ which preserve the required subspaces, the map $C$ such that $(A \otimes B \otimes C) S = S$ always exists and is unique.
To prove this, note that $S$ considered as a linear map $U_3^* \to U_1 \otimes U_2$ is an isomorphism between $U_3^*$ and the subspace $(\langle e_1 \dots e_{u_1 - 1} \rangle \otimes U_2 \oplus \left<e_{u_1 }\right> \otimes \left<e_1, \dots, e_{u_2  - p}\right>) \subset U_1 \otimes U_2$
Thus, $C$ can be found as the contragredient map to $A \otimes B$ restricted to this subspace.

From the description of $\operatorname{Stab}_G(S)$ it follows that
\[
\dim \operatorname{Stab}_G(S) = (u_1^2 - u_1 + 1) + (u_2^2 - p(u_2 - p))
\]
and 
\begin{multline*}
\dim G \cdot S = u_3^2 + (u_1 - 1) + p(u_2 - p) = u_3(u_1 u_2 - p) + (u_1 - 1) + p(u_2 - p) =\\ u_1 u_2 u_3 - p (u_3 - u_2) + u_1 - 1 - p^2 < u_1 u_2 u_3 - u_3 + u_2 + u_1 - 2 < u_1 u_2 u_3.
\end{multline*}
The last inequality holds because $u_2$ cannot be equal to $1$ under the assumptions of the theorem, and thus $u_3 \geq (u_1 - 1) u_2 > u_1 - 2 + u_2$.

It follows that the orbits of $T$ and $S$ are distinct and thus, $T \not\geq S$ by~\autoref{lem:concisesameorbit}.
\end{proof}

The proof of~\autoref{thm:prehompartialdeg} can be used to find concrete examples for partial degenerations: In fact, the proof of~\autoref{thm:prehomrange} gives a recursive way of constructing elements with dense orbit. A closed formula for elements of $U_1 \otimes U_2 \otimes U_3$ that have dense orbit on the other hand is not known. To see more concrete examples of partial degenerations, we now focus on tensors $T \in \mathbb{C}^{2}\otimes \mathbb{C}^{m}\otimes \mathbb{C}^{n}$. Clearly,~\autoref{thm:prehomrange} tells us that this space is $\GL(\mathbb{C}^m) \times \GL(\mathbb{C}^n)$-prehomogeneous whenever $m \neq n$. Fixing as basis $e_1, e_2$ of $\mathbb{C}^2$, write $T = e_1 \otimes T_1 + e_2 \otimes T_2$ where $T_1,T_2 \in \mathbb{C}^m \otimes \mathbb{C}^n$ can be thought of as $m \times n$ matrices. Hence $T$ is uniquely determined by a pair of matrices $[T_1,T_2]$, one often called the \emph{matrix pencil} associated with $T$. 

The fact that matrix pencil spaces are prehomogeneous has been known for a long time. In particular, we know an explicit element whose orbit is dense from~\cite{Pok:PerturbationsEquivalenceOrbitMatrixPencil}.

\begin{theorem}\label{thm:denseorbit} 
	For $m<n$ the action of $\GL(\mathbb{C}^m) \times \GL(\mathbb{C}^n)$ on $\mathbb{C}^{2}\otimes \mathbb{C}^{m}\otimes \mathbb{C}^{n}$ is \emph{prehomogeneous}, that is, there is a tensor $T \in \mathbb{C}^{2}\otimes \mathbb{C}^{m}\otimes \mathbb{C}^{n}$ such that its orbit closure is the whole space:
\begin{equation*}
	\overline{(\GL(\mathbb{C}^m) \times \GL(\mathbb{C}^n)).T} = \mathbb{C}^{2}\otimes \mathbb{C}^{m}\otimes \mathbb{C}^{n}
\end{equation*}
In particular, for any tensor $T$ with dense orbit, every other tensor is a partial degeneration of $T$. A particular choice of a tensor $T$ with dense orbit is the tensor associated with the matrix pencil $[I_1,I_2]$ where 
\begin{equation}\label{eq:denseorbitpencil}
I_1 =
    \begin{pmatrix}
    1&0&0&\dots &0 &0&\dots&0\\
    0&1&0&\dots &0 &0&\dots&0\\
    0&0&1&\dots &0 &0&\dots&0\\
    \vdots&\vdots&\vdots&\ddots &\vdots&\vdots&\ddots &\vdots\\
    0&0&0&\dots &1&0&\dots&0
    \end{pmatrix}, \;
    I_2 =
    \begin{pmatrix}
    0&\dots&0&1&0&0&\dots &0 \\
    0&\dots&0&0&1&0&\dots &0 \\
    0&\dots&0&0&0&1&\dots &0 \\
    \vdots&\ddots &\vdots&\vdots&\vdots&\vdots&\ddots &\vdots\\
    0&\dots&0&0&0&0&\dots &1
    \end{pmatrix}
\end{equation}
\end{theorem}

For example, let $n = m + 1$. According to~\autoref{thm:prehompartialdeg}, the pencil 
\begin{equation*}
    S_1 =
    \begin{pmatrix}
    1&0&0&\dots &0 &0\\
    0&1&0&\dots &0 &0\\
    0&0&1&\dots &0 &0\\
    \vdots&\vdots&\vdots&\ddots &\vdots&\vdots\\
    0&0&0&\dots &1&0
    \end{pmatrix}, \;
    S_2 =
    \begin{pmatrix}
    0&\dots&0&1 \\
    0&\dots&0&0 \\
    0&\dots&0&0 \\
    \vdots&\ddots &\vdots&\vdots\\
    0&\dots&0&0
    \end{pmatrix}
\end{equation*}
is an honest partial degeneration of the dense orbit element in~\autoref{eq:denseorbitpencil}.

\subsection{A no-go result for the unit tensor}\label{subsec:A no-go result for partial degenerations}

In this section, we will see that under certain circumstances, partial degenerations do not exist even when degenerations do. Recall that we defined the unit tensor as $\langle r \rangle = e_1 \otimes e_1 \otimes e_1 + \dots + e_r \otimes  e_r \otimes e_r$. We first show that there are no proper partial degenerations of the unit tensor if the constant map has full rank. This no-go result will be used to prove a classification result in this setting.

\begin{proposition}\label{prop:nogoresult}
Let $S \in V_1 \otimes V_2 \otimes V_3$ be any tensor. If $\langle r \rangle \bgeq S$ via degeneration maps $A_1, A_2(\epsilon)$ and $A_3(\epsilon)$ where the constant map $A_1$ is of full rank $r$ then $\langle r \rangle \geq S$.
\end{proposition}
\begin{proof}
It is clear that we can assume $\dim (V_1) = r$ and that $A_1$ is invertible.

Assume 
\begin{equation*}
    S = \lim_{\epsilon \rightarrow 0}(\id_{V_1} \otimes A_2(\epsilon) \otimes A_3(\epsilon)) \langle r \rangle 
\end{equation*}
is a degeneration where the first map is the identity. That is, we have
\begin{equation}\label{eq:decompSrankonematrices}
	S = e_1 \otimes M_1 + \dots  + e_r \otimes M_r
\end{equation}
where $M_i = \lim_{\epsilon \rightarrow 0} A_2(\epsilon)e_i \otimes A_3(\epsilon) e_i$. Hence, it is clear that for all $i$, $M_i$ must be a rank-1 matrix as limit of rank-1 matrices. But clearly, a tensor of the form in~\autoref{eq:decompSrankonematrices}, where all matrices $M_i$ have rank 1, is a restriction of $\langle r \rangle$.   

Now, let $S = \lim_{\epsilon \to 0} (A_1 \otimes A_2(\epsilon) \otimes A_3(\epsilon)) \langle r \rangle$ be any partial degeneration of $\langle r \rangle$. From before, we know that $\tilde{S}=(A_1^{-1} \otimes \id_{V_2} \otimes \id_{V_3})S =  \lim_{\epsilon \to 0} (\id_{V_1} \otimes A_2(\epsilon) \otimes A_3(\epsilon)) \langle r \rangle$ is a restriction of $\langle r \rangle$. Hence, the same holds for $S= (A_1 \otimes \id_{V_2} \otimes \id_{V_3} )\tilde{S}$. This finishes the proof. 
\end{proof}

\begin{remark}
	We note that the result in~\autoref{prop:nogoresult} does not apply to degenerations.~\autoref{prop:nogoresult} in particular says that if $V_1$ has dimension $r$ and $S \in V_1 \otimes V_2 \otimes V_3$ concise we cannot have an honest partial degeneration $\langle r \rangle \bgeq S$ (otherwise the constant map would be invertible by conciseness of $S$). But, for example, it is well-known that the unit tensor $\langle 2 \rangle$ does not restrict but degenerates to $W = e_1 \otimes e_1 \otimes e_2 + e_1 \otimes e_2 \otimes e_1 + e_2 \otimes e_1 \otimes e_1$ which is concise in the same space as $\langle 2 \rangle$. Hence, $W$ is an honest degeneration of $\langle 2 \rangle$ but not a partial degeneration. We note that the same holds for the degenerations $\langle 2^k \rangle \trianglerighteq W^{\boxtimes k}$ for all $k$. 
\end{remark}

It is clear that one cannot drop the condition that $A_1$ has full rank: in~\autoref{Strassenexample}, we saw that Strassen's tensor $\operatorname{Str}_r$ is an example of a partial degeneration of $\langle r \rangle$ where $A_1$ has rank $r-1$. In fact, we can use~\autoref{prop:nogoresult} to prove the following characterization of all partial degenerations of $\langle r \rangle$ where the constant map has rank $r-1$.
\begin{proposition}\label{prop:strassentensor}
    Let $T \in U_1 \otimes U_2 \otimes U_3$ with $\dim (U_1) = r - 1$ be a concise tensor such that $\langle r \rangle \bgeq T$ and $\langle r \rangle \not\geq T$.
    Then, for some $q$ such that $3 \leq q \leq r$, the tensor $T$ decomposes as
    \begin{equation*}
        T = S_q + X_{r-q}
    \end{equation*}
    where $\operatorname{Str}_{q} \geq S_q$ and $\langle r - q \rangle \geq X_{r-q}$. 
\end{proposition}
\begin{proof}
Suppose $\langle r \rangle \bgeq T$ via a partial degeneration
\begin{equation*}
    \lim_{\epsilon \rightarrow 0} \frac{1}{\epsilon^d}(A_1 \otimes A_2(\epsilon) \otimes A_3 (\epsilon)) \langle r \rangle  = T \in U_1 \otimes U_2 \otimes U_3.
\end{equation*}
Since $T$ is concise, the map $A_1$ has rank equal to $\dim (U_1) = r - 1$.
Note that $A_1$ can be factored as $A_1 = A M_q D P$ where
$A \colon \mathbb{C}^{r - 1} \to U_1$ is invertible,
$M_q\colon \mathbb{C}^{r}\rightarrow\mathbb{C}^{r-1}$ is defined as
\begin{equation*}
    M_q \colon \begin{cases}
    e_i \mapsto e_i &\text{ for } 1\leq i \leq r-1 \\
        e_r \mapsto e_1 + \dots + e_{q - 1}
    \end{cases}
\end{equation*}
with $1 \leq q \leq r$,
$D \colon \mathbb{C}^r \to \mathbb{C}^r$ is diagonal,
and $P \colon \mathbb{C}^r \to \mathbb{C}^r$ is a permutation matrix.
Indeed, suppose $\pi \in \mathfrak{S}_r$ is a permutation such that $A_1 e_{\pi (1)} \vvirg A_1 e_{\pi (r-1)}$ are linearly independent and $A_1 e_{\pi(r)} = \lambda_1 A_1 e_1 + \dots + \lambda_{q - 1} A_1 e_{q - 1}$ with nonzero $\lambda_1, \dots, \lambda_{q - 1}$.
Defining $A$ to be the map $A \colon e_i \mapsto \lambda_i A_1 e_{\pi(i)}$, $D = \operatorname{diag}(\lambda_1^{-1}, \dots, \lambda_{q-1}^{-1}, 1, \dots, 1)$, and $P$ the permutation matrix corresponding to $\pi^{-1}$, we get the required factorization.

Note that $(DP \otimes \id \otimes \id) \langle r \rangle = (\id \otimes DP^{-1} \otimes P^{-1}) \langle r \rangle$.
Now, we can rearrange the partial degeneration $(A_1 \otimes A_2(\epsilon) \otimes A_3(\epsilon)) \langle r \rangle$ as
\[
(A_1 \otimes A_2(\epsilon) \otimes A_3(\epsilon)) \langle r \rangle = (A \otimes \id \otimes \id) (\id \otimes A_2(\varepsilon) D P^{-1} \otimes A_3(\varepsilon) P^{-1}) (M_q \otimes \id \otimes \id) \langle r \rangle 
\]
This means that if $\langle r \rangle \bgeq T$, then up to a change of basis $T$ is a partial degeneration of 
\[ (M_q \otimes \id \otimes \id) \langle r \rangle = \sum_{i = 1}^{q-1} e_{i} \otimes (e_{i} \otimes e_i + e_r \otimes e_r) + \sum_{i = q}^{r - 1} e_i \otimes e_i \otimes e_i
\]
with identity map on the first factor.

Define $H_q = \sum_{i = 1}^{q-1} e_{i} \otimes (e_{i} \otimes e_i + e_q \otimes e_q) \in \mathbb{C}^{q - 1} \otimes \mathbb{C}^q \otimes \mathbb{C}^q$.
We have $(M_q \otimes \id \otimes \id) \langle r \rangle \simeq H_q \oplus \langle r - q \rangle$.
Using~\autoref{prop:nogoresult}, we see that $T = S_q + X_{r-q}$ where $S_q$ is a partial degeneration of $H_q$ and $X_{r-q}$ is a restriction of $\langle r-q \rangle$.
It remains to analyze partial degenerations of $H_q$.

So, consider a partial degeneration 
\begin{equation*}
    S_q = \lim_{\epsilon \rightarrow 0}\frac{1}{\epsilon^d}(\id \otimes B(\epsilon ) \otimes C(\epsilon ))H_q.
\end{equation*}
Define $b_i(\epsilon) = B(\epsilon) e_i$ and $c_i(\epsilon ) = C(\epsilon ) e_i$.
Suppose that $b_q(\epsilon) = b_{q, \mu} \epsilon^{\mu} + b_{q, \mu + 1} \epsilon^{\mu + 1} + \dots$. After a basis change we may assume that $b_{q, \mu} = e_q$.
Define
\[
E(\epsilon) \colon \begin{cases} e_i \mapsto e_i, & i < q \\ e_q \mapsto \epsilon^{-\mu} b_q(\epsilon) \end{cases}
\]
We have $\lim_{\epsilon \to 0} E(\epsilon) = \id$, so by changing $B(\epsilon)$ to $E(\epsilon)^{-1} B(\epsilon)$ we obtain a partial degeneration for the same tensor $S_q$ with $b_q(\epsilon) = \epsilon^{\mu} e_q$.
Using the same argument, we can assume without loss of generality that $c_q(\epsilon) = -\epsilon^\nu e_q$.
In this situation, we have
\begin{equation*}
    S_q = \lim_{\epsilon \rightarrow 0} \frac{1}{\epsilon^{d}}(\id \otimes B(\epsilon ) \otimes C(\epsilon )) H_p = \sum_{i = 1 }^{q-1} e_i \otimes \left( \frac{1}{\epsilon^{d}}b_i (\epsilon) \otimes c_i(\epsilon) - \epsilon^{\lambda + \mu - d} e_q \otimes e_q \right).
\end{equation*}
If $\lambda + \mu > d$, we clearly have 
\begin{equation*}
    S_q =\lim_{\epsilon \rightarrow 0}\sum_{i = 1 }^q e_i \otimes \left( \frac{1}{\epsilon^{d}}b_i (\epsilon) \otimes c_i(\epsilon) \right).
\end{equation*}
In this case, $S_q$ is a partial degeneration of $\langle q - 1 \rangle$ and
by~\autoref{prop:nogoresult}, we can choose the $b_i (\epsilon)$ and $c_i(\epsilon)$ constant in $\epsilon$ and obtain $\langle q-1 \rangle \geq S_q$  which yields $T \leq \langle r-1 \rangle \leq \langle r \rangle$.

If $\lambda + \mu < d$, we must have 
\begin{equation*}
    \begin{array}{c}
         b_i(\epsilon) = \epsilon^\sigma e_0 + \Tilde{b_i}(\epsilon)  \\
          c_i(\epsilon) = \epsilon^\tau e_0 + \Tilde{c_i}(\epsilon) 
    \end{array}
\end{equation*}
with $\sigma + \tau = \lambda + \mu$ so that 
\begin{equation*}
    S_q = \lim_{\epsilon \rightarrow 0} \frac{1}{\epsilon^{d}}\sum_{i = 1}^{q-1} e_i \otimes \left( \epsilon^{\sigma}\Tilde{b_i}(\epsilon) \otimes e_q + \epsilon^{\tau} e_q \otimes  \tilde{c_i}(\epsilon) \right).
\end{equation*}
For each $i = 1, \dots , q-1$, the limit 
\begin{equation*}
   e_i ^* S_q =  \lim_{\epsilon \rightarrow 0} \frac{1}{\epsilon^{d}} \left( \epsilon^{\sigma}\Tilde{b_i}(\epsilon) \otimes e_q + \epsilon^{\tau} e_q \otimes  \tilde{c_i}(\epsilon) \right) 
\end{equation*}
must exist and is of the form $b_i \otimes e_q + e_q \otimes c_i$ for some $b_i \in U_2$ and $c_i \in U_3$. Consequently, $S_q = \sum_{i = 1}^{q - 1} e_i \otimes (b_i \otimes e_q + e_q \otimes c_i)$ is a restriction of $\operatorname{Str}_q$.

Finally, consider the case $\lambda + \mu = d$. Here it holds that
\begin{align*}
    S_q &= \lim_{\epsilon \rightarrow 0} \sum_{i = 1 }^{q-1} e_i \otimes \left( \frac{1}{\epsilon^{d}}b_i (\epsilon) \otimes c_i(\epsilon) - \epsilon^{\lambda + \mu - d} e_q \otimes e_q \right) = \\
    &= \lim_{\epsilon \rightarrow 0}\frac{1}{\epsilon^{d}} \sum_{i = 1}^{q-1} e_i \otimes  (b_i (\epsilon) \otimes c_i(\epsilon)) - \left(\sum_{i = 1 }^{q-1} e_i \right) \otimes  e_q \otimes e_q
\end{align*}
In this case, $S_q$ is a partial degeneration of $\langle q \rangle$, and applying~\autoref{prop:nogoresult}, we see that $\langle q \rangle \geq S_q$ and $\langle r \rangle \geq T$. 

We obtain that the only case where $\langle r \rangle \not \geq T$ is when $T = S_q + X_{r - q}$ with $S_q \leq \operatorname{Str}_q$ and $X_{r - q} \leq \langle r - q \rangle$ for some $q$ such that $1 \leq q \leq r$.
We can exclude cases $q = 1$ and $q = 2$ because in these cases $\operatorname{Str}_q \leq \langle q \rangle$.
\end{proof} 

\section{Aided restriction and aided rank}\label{sec:Aided rank}
A related notion to partial degeneration is the notion of aided rank which we will introduce in~\autoref{subsec:Interpolation with a matrix}. There, we will also explain its relation to partial degenerations.

In~\autoref{subsec:A substitution method for aided rank}, we will present a generalization of the method to lower bound rank in~\cite{5959837} and use it in~\autoref{subsub:W} to calculate the aided rank for tensor powers of the $W$-tensor.

\subsection{Aided restriction and interpolation}\label{subsec:Interpolation with a matrix}

In this section, we will introduce the notion of aided rank and show its relation to partial degeneration. For any tensor $T \in U_1 \otimes  U_2 \otimes U_3$ recall the notation
\begin{equation*}
	T^{\thickdot p} = T \boxtimes \langle 1,1,p \rangle.
\end{equation*}

Recall the following interpolation result, which is based on ideas introduced in~\cite{BiCaLoRo:O277ComplexityApproximateMatMult}. 

\begin{theorem}\label{thm:interpolatingasusual}
Let $T\in U_1 \otimes U_2 \otimes U_3$ and $S \in V_1 \otimes V_2 \otimes V_3$ such that $T \trianglerighteq_{d}^{e} S$. Then, $T \boxtimes \langle e + 1 \rangle \geq S$ and $T \boxtimes \langle 2d + 1 \rangle \geq S$. 
\end{theorem}

We start by observing that one can use a unit matrix instead of a unit tensor to interpolate degenerations. We use notation from matrix multiplication in order to write this matrix as $\langle 1,1,p \rangle$ where $p$ is the rank of the unit matrix.

\begin{lemma}\label{lem:interpolatewithmatrixanydeg}
Consider tensors $T \in U_1 \otimes U_2 \otimes U_3$ and $S \in V_1 \otimes V_2 \otimes V_3$ and assume $T\trianglerighteq S$. Then,
\begin{equation*}
T^{\thickdot   u_3 v_3 }\geq S.
\end{equation*}
\end{lemma}
\begin{proof}
Fixing bases of the vector spaces involved, write
\begin{equation*}
	T^{ \thickdot u_3} = \sum_{i,j,k,l} T_{i,j,k} e_i \otimes (e_j \otimes e_l ) \otimes (e_k \otimes e_l )
\end{equation*}
where $i = 1, \dots, u_1$, $j = 1, \dots, u_2$ and $k,l = 1, \dots ,  u_3$.

Letting $\Pi_{3}:U_3 \otimes U_3 \rightarrow \mathbb{C}$ be the linear functional that maps $e_j \otimes e_l$ to 1 if $j = l$ and $0$ otherwise, we see by applying $\Pi_{U_3}$ on the third tensor factor that 
\begin{equation*}
	T^{\thickdot u_3} \geq \sum_{i,j,k} t_{i,j,k} e_i \otimes (e_j \otimes e_k) = \tilde{T}.
\end{equation*}

The tensor $\tilde{T}$ is the tensor $T$ seen as a bipartite tensor in $U_1 \otimes (U_2 \otimes U_3)$. We can also interpret $S$ as bipartite tensor $\tilde{S} \in V_1 \otimes (V_2 \otimes V_3)$. By assumption $\tilde{T}\trianglerighteq \tilde{S}$. In fact, since degeneration and restriction are equivalent for tensors on two factors, we deduce $\tilde{T}\geq \tilde{S}$. 

Again fixing bases, we have 
\begin{equation*}
	\tilde{S}^{\thickdot v_3} = \sum_{i,j,k,l} S_{i,j,k} e_i \otimes (e_j \otimes e_k \otimes  e_l ) \otimes  e_l.
\end{equation*}

As before, we can now define $\Pi_{V_3}:V_3 \otimes V_3 \rightarrow \mathbb{C}$ which maps $e_k \otimes e_l$ to 1 if $k = l$ and $0$ otherwise.
Applying $\id\otimes \Pi_{V_3}$ to the second tensor factor we see $ \tilde{S}^{\thickdot v_3}\geq S$. 

After all, we have seen 
\begin{equation*}
	T^{\thickdot (u_3 \cdot v_3)}= \left( T^{\thickdot u_3} \right)^{\thickdot v_3} \geq  \tilde{T}^{\thickdot v_3}\geq \tilde{S}^{\thickdot v_3}\geq S.
\end{equation*}
\end{proof}
\begin{remark}
The proof technique for~\autoref{lem:interpolatewithmatrixanydeg} is inspired from the physical interpretation of tensors: Considering the tensors as three party quantum states, we used two EPR-pairs to \emph{teleport} the \emph{particle} at the third party to the second party and back.
\end{remark}

The main question we ask is for a degeneration $T \trianglerighteq S$, how big must $p$ be such that $T^{\thickdot p} \geq S$. We will see that the minimal $p$ necessary to turn the degeneration into a restriction $T^{\thickdot p}\geq S$ can be chosen drastically smaller if the degeneration is a partial degeneration. On the other hand, we will calculate $p$ precisely for the degeneration $\langle 2^k \rangle \trianglerighteq W^{\boxtimes k}$ where we know from~\autoref{prop:nogoresult} that no partial degeneration exists. As it will turn out, here the minimal $p$ differs from the naive bound in~\autoref{lem:interpolatewithmatrixanydeg} only by a factor of $\frac{1}{2}$. To simplify further discussions, let us introduce the following definition.

\begin{definition}\label{def:aidedrank}
Let $S \in V_1 \otimes V_2 \otimes V_3$ and fix $p \geq 1$. We define the $p$-th \emph{aided rank} of $S$ as 
\begin{equation*}
R^{\thickdot p}(T) = \min\lbrace r \colon \langle r \rangle^{\;\thickdot p} \geq S\rbrace.
\end{equation*}
\end{definition}

Clearly, we have $R^{\thickdot 1}(T) = R(T)$.~\autoref{lem:interpolatewithmatrixanydeg} shows that $\underline{R}(S) =r$ implies that there is some $q$ such that $R^{\thickdot p}(S) = r$. To find better bounds on the minimal $p$, we now show a variation of~\autoref{thm:interpolatingasusual}.

\begin{proposition}\label{prop:interpolate}
Let $T \in U_1 \otimes U_2 \otimes U_3$ and $S \in V_1 \otimes V_2 \otimes V_3$ and assume $T \bgeq_d^e S$. Then, 
\begin{equation*}
T^{\thickdot d+1} \geq S \text{ and } T^{\thickdot e+1} \geq S.
\end{equation*}
\end{proposition}
\begin{proof}
Suppose that the partial degeneration is given by
	\begin{equation}\label{eq:somedeg}
    \left(A_1 \otimes A_2(\epsilon) \otimes A_3(\epsilon)\right)T = \epsilon^d S + \sum_{i = 1}^e \epsilon^{d + i} S_i.
\end{equation}
Powers of $\epsilon$ higher than $d$ have no effect on the coefficient of $\eps^d$; hence without loss of generality, we may assume  
\begin{equation*}
	A_2(\epsilon) = \sum_{i = 0}^{d} \epsilon^i A_{2,i},\;\;\; A_3(\epsilon) = \sum_{i = 0}^{d} \epsilon^i A_{3,i}.
\end{equation*}
We then observe 
\begin{equation*}
	S = A_1 \otimes \left(\sum_{i = 0}^{d} A_{2,i} \otimes A_{3,d - i} \right) T
\end{equation*}
and therefore 
\begin{equation*}
	A_1 \otimes \left(\sum_{i = 0}^{d} A_{2,i} \otimes e_i^* \right) \otimes \left(\sum_{i = 0}^{d} A_{3,d-i} \otimes e_i^* \right) T^{\thickdot d+1} = S
\end{equation*}
which shows $T^{\thickdot d+1}\geq S$. 

In order to see $T^{\thickdot e+1}\geq S$, note that for $\epsilon > 0$, we can rewrite~\autoref{eq:somedeg} as 
\begin{equation*}
\left(A_1 \otimes (A_2(\epsilon)/\epsilon^d)\otimes A_3(\epsilon)\right)T = S + \epsilon S_1 + \dots + \epsilon^e S_e =\colon q(\epsilon).
\end{equation*}
Using Lagrangian interpolation we can pick $\alpha_0, \dots , \alpha_e \neq 0$ such that
\begin{equation*}
    q(\epsilon) = \sum_{j = 0}^{e} q(\alpha_j) \prod_{m \neq j} \frac{\epsilon - \alpha_m}{\alpha_j - \alpha_m}.
\end{equation*}
By writing $\mu_j \colon = \prod_{m \neq j}\frac{\alpha_m}{\alpha_m - \alpha_j}$, we therefore get $S = q(0) = q(\alpha_0)\mu_0 + \dots + q(\alpha_e)\mu_e$.
Note that the $q(\alpha_j)$ are all restrictions of $T$ where the first restriction map can be chosen to be $A_1$. With that, 
\begin{equation*}
	S = q(0) =\left(A_1\otimes ( \sum_{j = 0}^{e} \frac{\mu_j}{\alpha_j^d}A_2(\alpha_j) \otimes e_j^*) \otimes (\sum_{j = 0}^{e} A_3(\alpha_j) \otimes e_j^*)\right) T^{\thickdot e+1}
\end{equation*}
which finishes the proof.
\end{proof}

In particular, we can exclude partial degeneration with a certain approximation degree if we give lower bounds on the aided rank of a tensor. Note that in the case of prehomogeneous spaces, we can find an even better bound.

\begin{proposition}\label{prop:denseorbitonlyneedsaidtwo}
Assume that $U_1 \otimes U_2 \otimes U_3$ is prehomogeneous under the action of $\text{GL}(U_2)\times \text{GL}(U_3)$ and let $T$ be an element with dense orbit. Then, for all $S \in U_1 \otimes U_2 \otimes U_3$ it holds that
\begin{equation*}
T^{\thickdot 2} \geq S.
\end{equation*}
\end{proposition}
\begin{proof}
Consider the affine degree-1 curve $L$ parameterized by $L(\epsilon) = T + \epsilon (S - T)$. It is clear that both $T$ and $S$ lie on $L$. Clearly, the linear span of any two distinct points on $L$ contains all points on $L$. The orbit of $T$ is open in $U_1 \otimes U_2 \otimes U_3$, therefore the intersection of $L$ and the complement of the orbit of $T$ is a closed subset of $L$, that is, a finite collection of points. Hence, there exists a second point $\tilde{T}$ in the orbit of $T$ on $L$. Writing $\tilde{T} = (\text{id}\otimes M_2 \otimes M_3)T$, and $S = \lambda T + \mu \tilde{T}$, we observe
\begin{equation*}
	\left[\id \otimes (\lambda \id \otimes e_1^* + \mu M_2 \otimes e_2^*) \otimes ( \id \otimes e_1^* +  M_3 \otimes e_2^*)\right] T^{\thickdot 2} = S
\end{equation*}
which proves the claim.
\end{proof}
Note that~\autoref{prop:denseorbitonlyneedsaidtwo} supports the intuition that in the case of partial degenerations, the minimal aiding rank $q$ turning it into a restriction is small. In~\autoref{subsec:Matrix pencils}, we saw that whenever $T \in U_1 \otimes U_2 \otimes U_3$ has a dense orbit under the action of $\GL(U_2)\times \GL(U_3)$ it holds for all $S \in U_1 \otimes U_2 \otimes U_3$ that $T \bgeq S$. 

We prove two variants of~\autoref{prop:denseorbitonlyneedsaidtwo} in the case where the $\GL(U_2) \times \GL(U_3)$ is not dense. For that, we use an argument introduced in \cite{ChrGesJen:BorderRankNonMult}, where one exploits the genericity of certain linear spaces intersecting the orbit to reconstruct elements on the linear space. In the proofs of~\autoref{prop:hypersurface} and~\autoref{prop:furtherthanhypersurface} we use some basic notation and results from algebraic geometry, we refer to~\cite{Harris:AlgGeo} for details. In particular, we refer to~\cite[Remark 4.4]{ChrGesJen:BorderRankNonMult} for a characterization of the degree of an algebraic variety in terms of intersection multiplicity for a generic line. 
\begin{proposition}\label{prop:hypersurface}
Let $T \in U_1 \otimes U_2 \otimes U_3$; let $\Omega_T = (\GL(U_2) \times \GL(U_3)) \cdot [T] \subseteq \bbP(U_1 \otimes U_2 \otimes U_3)$ be the orbit of the point $[T]$ and let $X_T = \bar{\Omega_T}$ be its Zariski-closure. Suppose $X_T$ is a hypersurface in $\bbP(U_1 \otimes U_2 \otimes U_3)$. Let $S \in U_1 \otimes U_2 \otimes U_3$ be an element such that $\mult_{X_T}([S]) \leq \deg(X_T) - 2$. Then
\[
T^{\thickdot 2} \geq S.
\]
\end{proposition}
\begin{proof}
Let $ m = \mult_{X_T}([S])$. In particular $m = 0$ if $[S] \notin X_T$. Let $L$ be a generic line through $[S]$. The genericity assumption guarantees that $X_T \cap L$ is a $0$-dimensional scheme of degree $\deg(X_T)$; by Bertini's Theorem  \cite[Sec. 1.1, p.137]{GrifHar:PrinciplesAlgebraicGeometry} this scheme has a component of degree $m$ supported at $[S]$ and $\deg(X_T) - m$ distinct points. Moreover, by genericity $L \cap ( X_T \setminus \Omega_T)$ is either empty or it contains the single point $[S]$, therefore all intersection points in $L \cap X_T$, except possibly $S$, belong to $\Omega_T$.

By assumption $\deg(X_T) - m \geq 2$, so there exist two distinct points $[T_1],[T_2] \neq S$ in $L \cap \Omega_T$. Since $[T_1],[T_2]$ are distinct, they span the line $L$; therefore there exist scalars $\lambda_1,\lambda_2 \in \bbC$ such that $S = \lambda_1 T_1 + \lambda_2 T_2$. Since $[T_1],[T_2] \in \Omega_T$, there exist $g_2^{(1)} \otimes g_3^{(1)},g_2^{(2)} \otimes g_3^{(2)} \in \GL(U_2) \times \GL(U_3)$ such that $T_i = g_2^{(i)} \otimes g_3^{(i)} T$.

Consider the restriction of $T^{\thickdot 2} = T \boxtimes \langle 1,1,2\rangle$ to $U_1 \otimes U_2 \otimes U_3$ defined by the maps
\begin{align*}
M_1 &= \id_{U_1}  : U_1 \to U_1 \\
M_2 &= \lambda_1 g_2^{(1)} \otimes e_1^* + \lambda_2 g_2^{(2)} \otimes e_2^* : U_2 \otimes \bbC^2 \to U_2 \\
M_3 &= g_3^{(1)} \otimes e_1^* + g_2^{(2)} \otimes e_2^* : U_3 \otimes \bbC^2 \to U_3.
\end{align*}
It is immediate that $M_1 \otimes M_2 \otimes M_3 ( T^{\thickdot 2}) = S$.
\end{proof}

\begin{proposition}\label{prop:furtherthanhypersurface}
Let $T \in U_1 \otimes U_2 \otimes U_3$ be concise in the first factor; let $\Omega_T = (\GL(U_2) \times \GL(U_3)) \cdot [T] \subseteq \bbP(U_1 \otimes U_2 \otimes U_3)$ be the orbit of the point $[T]$ and let $X_T = \bar{\Omega_T}$ be its Zariski-closure. Let $c = \codim X_T$ be the codimension of $X_T$ in $\bbP(U_1 \otimes U_2 \otimes U_3)$. Let $S \notin X_T$. Then
\[
T^{\thickdot (c+1)} \geq S.
\]
\end{proposition}
\begin{proof}
Since $T$ is concise in the first factor, the variety $X_T$ is not linearly degenerate. In particular $\deg(X_T) \geq c+1$  \cite[Corollary 18.12]{Harris:AlgGeo}. Let $L$ be a generic linear space through $S$ of dimension $c$. Bertini's Theorem, together with the genericity assumption, guarantees that $L \cap X_T$ is a set of $\deg(X_T)$ distinct points and by genericity they all belong to $\Omega_T$. Moreover, it is easy to see that $L \cap X_T$ is not linearly degenerate in $L$. Therefore, there exist $[T_0] \vvirg [T_c] \in L \cap X_T$ that span $L$, and in particular there exist $\lambda_0 \vvirg \lambda_c$ such that $S = \lambda_0 T_0 + \cdots + \lambda_c T_c$.

Since $T_j \in \Omega_T$ for every $j$, there exists $g^{(j)}_2 \otimes g^{(j)}_3 \in \GL(U_2) \times \GL(U_3)$ such that $g^{(j)}_2 \otimes g^{(j)}_3  T = T_j$. 

Consider the restriction of $T^{\thickdot (c+1)} = T \boxtimes \langle 1,1,c+1\rangle$ to $U_1 \otimes U_2 \otimes U_3$ defined by the maps
\begin{align*}
M_1 &= \id_{U_1}  : U_1 \to U_1 \\
M_2 &= \sum_{j=0}^c ( \lambda_j g_2^{(j)} \otimes e_j^* ) : U_2 \otimes \bbC^{c+1} \to U_2 \\
M_3 &= \sum_{j=0}^c ( g_3^{(j)} \otimes e_j^* ) : U_3 \otimes \bbC^{c+1} \to U_3. 
\end{align*}
It is immediate that $M_1 \otimes M_2 \otimes M_3 ( T^{\thickdot (c+1)}) = S$.
\end{proof}

\subsection{A substitution method for aided rank}\label{subsec:A substitution method for aided rank}

In this section, we will give a method to bound from below the aided rank. Our method builds on a known method from~\cite{5959837}. We will use it to calculate aided ranks of powers of the $W$-tensor. We start by mentioning the following easy technical fact without proof.
\begin{lemma}\label{lem:vectorspacedims}
Let $V$ be a vector space and $U$ be a subspace of dimension $d$ of $V$. If $U$ is contained in the span of vectors $u_1, \dots , u_d$, then all $u_i$ must be elements of $U$.
\end{lemma}

The second lemma gives a useful characterization of restrictions of $\langle n \rangle^{\thickdot p}$ in terms of flattenings. It is a simple generalization of a well-known characterization of tensor rank, see for example~\cite[Theorem 3.1.1.1]{Lan:TensorBook}
\begin{lemma}\label{lem:aidedrankchar}
Let $S \in V_1 \otimes V_2 \otimes V_3$ be any tensor and fix some natural number $p$. Then we have
\begin{equation*}
	R^{\thickdot p}(S) = \min \lbrace r : S\left(V_1^*\right) \subseteq V_2 \otimes V_3 \text{ spanned by } r \text{ matrices of rank } \leq p \rbrace.
\end{equation*}
\end{lemma}

\begin{proof}
If  $\langle r \rangle^{\thickdot p} \geq S$ we can write $S = a_1 \otimes M_1 + \dots + a_r \otimes M_r$ for matrices $M_i$ of rank at most $p$, in other words, $S(V_1^*)$ is spanned by $r$ matrices $M_1, \ldots , M_r$ of rank at most $p$.

On the other hand, assume $S(V_1^*) = \langle N_1, \ldots ,N_r \rangle$ for matrices $N_i$ of rank at most $p$. Fixing a basis of $V_1$ the tensor $S$ is given by $\sum_{i=1}^{v_1} e_i \otimes P_i$ where $P_i = S(e^i)$. Since $S(V_1^*)$ is spanned by the $N_j$, we can find coefficients $\lambda_{ij}$ such that $P_i = \sum_{j=1}^r \lambda_{ij}N_j$. Hence, 
\begin{equation*}
	S = \sum_{j = 1}^{r} \left( \sum_{i=1}^{v_1} \lambda_{i,j}e_i \right) \otimes N_j,
\end{equation*}
which expresses $S$ as a restriction of $\langle r \rangle^{\; \thickdot p}$. This concludes the proof.
\end{proof}

We can use~\autoref{lem:aidedrankchar} to generalize a method to lower bound rank of tensors introduced in~\cite{5959837} . 

\begin{theorem}\label{thm:aidedrankreduction}
Let $S \in V_1\otimes V_2 \otimes V_3$ with $\dim(V_1) = v_1$. Fix a basis $e_1, \ldots ,e_{v_1}$ of $V_1$, and write
\begin{equation*}
    S = \sum_{i = 1}^{v_1} e_i  \otimes M_i
\end{equation*}
for matrices $M_i \in V_2 \otimes V_3$ and assume $M_1 \neq 0$. Moreover, for complex numbers $\lambda_2, \dots , \lambda_{v_1}$ define \begin{equation*}
    \hat{S}(\lambda_2, \dots , \lambda_{v_1}) = \sum_{j = 2}^{v_1} e_j \otimes (M_j - \lambda_j M_1).
\end{equation*}
We have the following:
\begin{enumerate}[(i)]
    \item There exist $\lambda_2, \dots , \lambda_{v_1}\in \mathbb{C}$ such that 
    \begin{equation*}
    R^{\thickdot p}(\hat{S}(\lambda_2, \dots , \lambda_{v_1})) \leq R^{\thickdot p}(S) - 1.
    \end{equation*}
    \item Assume that $M_1$ has rank at most $p$. Then, for all $\lambda_2, \dots , \lambda_{v_1}$
    \begin{equation*}
    R^{\thickdot p}(\hat{S}(\lambda_2, \dots , \lambda_{v_1})) \geq R^{\thickdot p}(S) - 1.
\end{equation*}
\end{enumerate}
Hence, if $M_1$ has rank at most $p$, we always find $\lambda_2, \dots , \lambda_{v_1}$ such that 
\begin{equation*}
    R^{\thickdot p}(\hat{S}(\lambda_2, \dots , \lambda_{v_1})) = R^{\thickdot p}(S) - 1.
\end{equation*}
\end{theorem}

\begin{proof}
Let $r = R^{\thickdot p}(T)$, that is, $S(V_1^*)$ is contained in the span of $r$ matrices of rank at most $p$. Denote these matrices by $X_1, \dots , X_r$ and write
\begin{equation*}
M_i = \sum_{j = 1}^r \mu_{i, j} X_{j} \text{ for }i = 1, \ldots ,v_1.
\end{equation*}
Without loss of generality assume that $\mu_{1,1} \neq 0$ and set $\lambda_j = \frac{\mu_{j,1}}{\mu_{1,1}}$. We easily see that $\hat{S}(\lambda_2, \dots, \lambda_a)(V_1^*) \subset \langle X_2, \dots , X_r \rangle$, and therefore 
\begin{equation*}
R^{\thickdot p}(\hat{S}(\lambda_2, \dots , \lambda_{v_1})) \leq R^{\thickdot p}(S) - 1.
\end{equation*}
That shows the first claim.\\
On the other hand, if $M_1$ has rank at most $p$ and $ Y_1, \dots , Y_s $ span $\hat{S}(\lambda_2, \dots, \lambda_{v_1})$, then clearly the matrices $M_1, Y_1, \dots , Y_s$ will span $S(V_1^*)$, which shows the second claim.
\end{proof}

 In the next section we will see how one can use~\autoref{thm:aidedrankreduction} to calculate aided ranks. 

\subsection[Aided rank of Kronecker powers of the W-tensor]{Aided rank of Kronecker powers of the $W$-tensor}\label{subsub:W}
Let $V_1 ,V_2 $ and $ V_3 $ be 2-dimensional with fixed bases $e_1,e_2$. In this section, we will use the method developed in~\autoref{subsec:A substitution method for aided rank} to calculate the aided rank of powers of the $W$-tensor 
\begin{equation*}
 W = e_1 \otimes e_1 \otimes e_2 + e_1 \otimes e_2 \otimes e_1 + e_2 \otimes e_1 \otimes e_1 \in V_1 \otimes V_2 \otimes V_3.
\end{equation*}
The main result of this section is the following.

\begin{proposition}
For the $k$'th Kronecker power of the $W$ tensor $W^{\boxtimes k} \in V_1^{\otimes k} \otimes V_2^{\otimes k} \otimes V_3^{\otimes k}$ it holds that 
\begin{equation*}
    R^{\thickdot p}(W^{\boxtimes k}) \begin{cases} = 2^k \text{ if } p \geq 2^k \\ > 2^k \text{ if } p < 2^k.
    \end{cases}
\end{equation*}
\end{proposition}
\begin{proof}
It is clear that for any $r < 2^k$, $\langle r \rangle ^{\thickdot p} \not \geq W^{\boxtimes k}$ independently of $p$. On the other hand, $W(V_1^*)$ is the span of $e_1 \otimes e_2 + e_2 \otimes e_1$ and $e_1 \otimes e_1$, in other words, $R^{\thickdot 2}(W) = 2$. Consequently, we know that $\langle 2^k \rangle^{\;\thickdot 2^k} \geq W^{\boxtimes k}$ for all $k$, in other words, $R^{\thickdot 2^k}(W^{\boxtimes k}) \leq 2^k$. 

We will now use~\autoref{thm:aidedrankreduction} to show that $\langle 2^k \rangle^{\;\thickdot 2^k-1} \ngeq W^{\boxtimes k}$ which will finish the proof. One can verify easily that -- thinking of the elements of $V_2^{\otimes k} \otimes V_3^{\otimes k}$ as $2^{k}\times  2^{k}$ matrices -- all matrices in $W^{\boxtimes k}((V_1^{\otimes k})^*)$ are of the form 
\begin{equation}\label{formofmatrix}
    \begin{pmatrix}
 * &  & x_0\\
  &  \iddots&\\
   x_0&  &0\\
\end{pmatrix}.
\end{equation}
That is, all matrices in $W^{\boxtimes k}((V_1^{\otimes k})^*)$ have the same entry $x_0$ in all antidiagonal entries and zeros in all entries below the antidiagonal. Now, assume for some $p$ that $\langle 2^k \rangle^{\thickdot p} \geq W^{\boxtimes k}$. By~\autoref{lem:aidedrankchar}, there are matrices $N_i$ of rank at most $p$ such that
\begin{equation}\label{eqcontainment}
    W^{\boxtimes k}\left((V_1^{\otimes k})^* \right) \subseteq \langle N_1 , \dots , N_{2^k} \rangle.
\end{equation}
As $W^{\boxtimes k}$ is concise, $W^{\boxtimes k}\left(V_1^{\otimes k})^* \right)$ has dimension $2^k$. Therefore, by~\autoref{lem:vectorspacedims}, the $N_i$ are elements of $W^{\boxtimes k}\left((V_1^{\otimes k})^* \right)$. We observe that a matrix of the form~\autoref{formofmatrix} with $x_0 \neq 0$ has full rank $2^k$. That is, if the matrices $N_i$ have rank $p < 2^k$ and are elements of $W^{\boxtimes k}\left((V_1^{\otimes k})^* \right)$, their span only contains matrices with zeros on the antidiagonal. That is,~\autoref{eqcontainment} cannot be satisfied if all $N_i$ have rank at most $p < 2^k$, that is, 
\begin{equation*}
    \langle 2^k \rangle^{\thickdot p} \ngeq W^{\boxtimes k} \text{ if } p<2^k.
\end{equation*}
In other words, $R^{\thickdot p}(W^{\boxtimes k}) > 2^k$ for $p< 2^k$.
\end{proof}

In particular we see that the minimal rank of an aiding matrix turning the degeneration $\langle 2^k \rangle \trianglerighteq W^{\boxtimes k}$ into a restriction differs from the bound in~\autoref{lem:interpolatewithmatrixanydeg} only by a factor of $\frac{1}{2}$. 

\bibliographystyle{alphaurl}
\bibliography{degConstMap.bib}

\newpage
\begin{appendix}
\appendixpage
\section{Code for~\autoref{rem:orbitdimof222}}\label{codeappendix}
The following Macaulay2~\cite{M2} code gives a lower bound of 37 on the dimension of the orbit of a generic tensor in $\mathbb{C}^3 \otimes \mathbb{C}^4 \otimes \mathbb{C}^4$ which is $(2,3,3)$-compressible. This code is an adjustment of the one used in \cite{https://doi.org/10.48550/arxiv.2101.03148}.

{\small
\begin{tcolorbox}[fontupper=\footnotesize]
\begin{verbatim}
V_1 = QQ[v_(1,1)..v_(1,3)]
V_2 = QQ[v_(2,1)..v_(2,4)]
V_3 = QQ[v_(3,1)..v_(3,4)]
W_1 = QQ[w_(1,1)..w_(1,3)]
W_2 = QQ[w_(2,1)..w_(2,4)]
W_3 = QQ[w_(3,1)..w_(3,4)]

ALL = V_1**V_2**V_3**W_1**W_2**W_3

M_1 = sub(random(QQ^4,QQ^4),ALL)
M_2 = mutableMatrix(ALL,4,4)
M_3 = mutableMatrix(ALL,4,4)

for i from 0 to 3 do(
    M_2_(0,i)=random(QQ);
    M_2_(i,0)=random(QQ);
    M_3_(0,i)=random(QQ);
    M_3_(i,0)=random(QQ);
    )
M_2 = matrix M_2
M_3 = matrix M_3

T = 0
for i from 1 to a do(
    for j from 1 to b do(
	        for k from 1 to c do(
	            T = T + M_i_(j-1,k-1)*w_(1,i)*w_(2,j)*w_(3,k);
	            );
	        );
    	)--T is now (2,3,3)-compressible with random entries
	    
-- a random point in Hom(W1,V1) + Hom(W2,V2) + Hom(W3,V3)
-- the rank of the differential of the parameterization map at randHom
-- will provide a lower bound on dim of the orbit closure of our tensor

randHom =flatten flatten apply(3,j-> 
toList apply(1..di_(j+1),i ->w_(j+1,i)=>sub(random(1,V_(j+1)),ALL)))

-- compute the image of the differential
-- LL will be a list of elements of multidegree (1,1,1), 
-- which are to be interpreted as elements of V1 \otimes V2 \otimes V3
-- generating the image of the differential of the parameterization map
LL = flatten for i from 1 to 3 list (
     ww = sub(vars(W_i),ALL);
     vv = sub(vars(V_i),ALL);
     flatten entries (sub( (vv ** diff(ww,Tused)),randHom)));
minGen = mingens (ideal LL);
orbitdim = numcols(minGen) --37
\end{verbatim}
\end{tcolorbox}
}
\section{Partial degenerations of the unit tensor}
We have seen in~\autoref{prop:nogoresult} that the unit tensor $\langle r \rangle$ does not admit partial degenerations where the constant map $A_1$ is full rank. However, we also saw that in the case that $A_1$ has rank $r-1$ there are honest partial degenerations which we classify in~\autoref{prop:strassentensor}. In this appendix, we see that also in the realm of matrix pencils, examples exist. For that, we consider for simplicity only matrix pencils in $\mathbb{C}^2 \otimes \mathbb{C}^m \otimes \mathbb{C}^{m + 1}$. Recall that the matrix pencil $T$ given in~\autoref{eq:denseorbitpencil} has a dense orbit under the action of $\GL_m \times \GL_{m+1}$. It is well known (see, for example, Theorem 3.11.1.1 in~\cite{Lan:TensorBook}) that this pencil has rank $m + 1$. On the other hand, it is known that the maximal rank of a tensor in $\mathbb{C}^2 \otimes \mathbb{C}^m \otimes \mathbb{C}^{m + 1}$ is $\lceil\frac{3m}{2}\rceil$. Hence, we can find tensors $S$ in $\mathbb{C}^2 \otimes \mathbb{C}^m \otimes \mathbb{C}^{m + 1}$ with $\langle m + 1 \rangle \geq T \bgeq S$, $R(S) > m+1$ and consequently $\langle m + 1 \rangle \bgeq S$ but  $\langle m+1 \rangle \ngeq S$. 

To see an explicit example, let us construct for every $m$ a matrix pencil of rank greater or equal to $m+1$ to which $\langle m \rangle$ degenerates partially. For this, we recall the following well-known result about the rank of matrix pencils~\cite{DBLP:conf/mfcs/Grigoriev78,DBLP:journals/siamcomp/JaJa79}.

\begin{proposition}\label{prop:rankofmatrixpencil}
Consider $p_1 \times  q_1$ matrices $T'_1, T'_2$ and $p_2 \times  q_2$ matrices $T''_1,T''_2$. Let $T'$ be the tensor corresponding to the matrix pencil $[T'_1,T'_2]$ and similar for $T''$ and write $T \in \mathbb{C}^2\otimes \mathbb{C}^{p_1 + p_2} \otimes \mathbb{C}^{q_1 + q_2}$ for the tensor corresponding to the matrix pencil
\begin{equation*}
    \left[ 
\begin{pmatrix}
	T'_1 & \\ &T''_1
\end{pmatrix},
\begin{pmatrix}
	T'_2 & \\ &T''_2
\end{pmatrix}
\right].
\end{equation*}
Then, it holds that
\begin{equation*}
R(T) = R(T') + R(T'').
\end{equation*}
\end{proposition}

We will now construct a partial degeneration of $\langle m \rangle$ and will show using~\autoref{prop:rankofmatrixpencil} that it has rank at least $m+1$. Applying the linear map 
\begin{equation*}
A_1 : U \rightarrow \mathbb{C}^{2}, e_k \mapsto e_1 + ke_2
\end{equation*}
we see that $\langle m \rangle$ restricts to the tensor corresponding with the matrix pencil $[\id_m,\mathrm{diag}(1 \dots  m)]$. Since the matrix 
\begin{equation*}
M = \begin{pmatrix}
	&1&1& 	   &	   &\\
	& &2&	  1&\\
	& & &\ddots &\ddots &\\
	& & &	   &	m-1&1\\
	& & &	   &       &m
\end{pmatrix}
\end{equation*}
has $m$ different eigenvalues $1,\dots ,m$, we deduce that also the tensor associated to the matrix pencil $[\id_m, M]$ is a restriction of $\langle m \rangle$.

For any $1<k<m$ define $S_{k,m} $ to be the tensor corresponding to the matrix pencil 
\begin{equation}\label{eq:oneconstdegofghz}
\left[  \begin{pmatrix}
 \begin{pmatrix}
     \mathrm{id}_{k-1}, & 0 
    \end{pmatrix} & 0 \\
    0 & \begin{pmatrix}
     0\\\mathrm{id}_{m-k} 
    \end{pmatrix} 
\end{pmatrix} ,\begin{pmatrix}
 J_1 & 0 \\
    0 & J_2
\end{pmatrix}  \right].
\end{equation}

where 

\begin{equation*}
    J_1 = 
    \begin{pmatrix}
     1 & 1&&  \\
       &\ddots&\ddots&\\
       &&k-1&1
    \end{pmatrix},
    J_2 = 
    \begin{pmatrix}
     1 & &&  \\
      k+1 &1&&\\
      &\ddots &\ddots &&\\
      &&m-1&1\\
       &&&m
    \end{pmatrix}
\end{equation*}

One verifies that applying the degeneration maps
\begin{equation*}
	A_2(\epsilon) = \mathrm{diag} (\underbrace{1,\dots , 1}_k,\underbrace{\epsilon,\dots ,\epsilon}_{m-k}), \;\; A_3(\epsilon) = \mathrm{diag} (\underbrace{\epsilon, \dots , \epsilon}_{k+1}, \underbrace{1,\dots ,1}_{m-k-1})
\end{equation*}

the tensor corresponding to the matrix pencil $[\id_k,\mathrm{diag}(1, \ldots , m)]$ results in $\epsilon S_{k,m} + \mathcal{O}(\epsilon^2)$. In particular, $\langle m \rangle \bgeq S_{k,m}$.

From~\autoref{prop:rankofmatrixpencil}, we know that
\begin{equation}\label{eq:pencilranksum}
	R(S_{k,m}) = R(S_{k,m}^1) + R(S_{k,m}^2)
\end{equation}
where $S_{k,m}^1$ corresponds to $[(\id_{k-1},0),J_1])$ and $S_{k,m}^2$ to $[\begin{pmatrix}
     0 &\mathrm{id}_{m-k} 
\end{pmatrix}^{T},J_2]$, respectively. Using flattenings, one can now verify that the two pencils in~\autoref{eq:pencilranksum} have ranks $k$ and $m-k +1$, respectively, which shows $R(S_{k,m})\geq m+1$. Hence, it is not a restriction of $\langle m \rangle$. 

\section{The aided rank of the Coppersmith-Winograd tensor}\label{subsec:aidedrankofCW}
In this appendix, we demonstrate with further examples how to compute aided rank using~\autoref{thm:aidedrankreduction}, we are going to calculate the aided ranks of the Coppersmith-Winograd (CW) tensors. The study of these tensors was a crucial tool in the breakthrough result~\cite{10.1145/28395.28396} bounding the exponent of matrix multiplication $\omega$ from above by 2.376.

\begin{definition}
Let $V_1 \simeq V_2 \simeq V_3 \simeq \mathbb{C}^{q+2}$ and fix a basis $e_0, \ldots ,e_{q+1}$. The $q$'th CW tensor is the symmetric tensor 
\begin{align*}
	T_{CW,q} = &\sum_{i = 1}^{q} e_0 \otimes e_i \otimes e_i + e_{q+1}\otimes e_{0} \otimes e_{0}+\\
		   &\sum_{j = 1}^{q} e_j \otimes e_0 \otimes e_j + e_{0}\otimes e_{q+1} \otimes e_{0}+\\
		   &\sum_{k = 1}^{q} e_k \otimes e_k \otimes e_0 + e_{0}\otimes e_{0} \otimes e_{q+1}\in V_1 \otimes V_2\otimes V_3.
\end{align*}
\end{definition}

We want to calculate $R^{\thickdot p}(T_{CW,q})$ for any $p$ and $q$. 

\begin{proposition}\label{prop:aidedrankcoppersmith}
For $p \geq 2$, the $p$-aided rank of the $q$'th Coppersmith-Winograd tensor is given by
\begin{equation*}
    R^{\thickdot p}(T_{CW,q}) = q+1 + \left\lceil \frac{q + 2}{p}  \right\rceil.
\end{equation*}
\end{proposition}
\begin{proof}
Writing 
\begin{equation*}
    M(x_0, \dots , x_{q+1}) = \begin{pmatrix}
    x_{q+1}  &  x_1  &  \dots   &  x_q  &  x_0  \\
    x_1      &  x_0  &          &       &       \\
    \vdots   &       &  \ddots  &       &       \\
    x_q      &       &          &  x_0  &       \\
    x_0      &       &          &       &  0
    \end{pmatrix}
\end{equation*}
we have $T_{CW,q}(V_1^*) = \lbrace M(x_0, \dots x_{q+1}):\; x_0, \dots x_{q+1} \in \mathbb{C} \rbrace$. Note that $T_{CW,q}$ is concise. Hence, we have $R^{\thickdot p}(T_{CW,q}) \geq q + 2$ for any $p \in \mathbb{N}$. Moreover, it is clear that $R^{\thickdot p}(T_{CW,q}) \leq q + 2$ whenever $p \geq q+2$ which gives $R^{\thickdot p}(T_{CW,q}) = q + 2$ for all $p \geq q+2$. 

For $p \leq q+1$, we will use~\autoref{thm:aidedrankreduction}. Suppose $p \geq 2$. Interpreting $V_2 \otimes V_3$ as space of $(q+2) \times (q+2)$ matrices, we have 
\begin{equation*}
    T_{CW,q} = \sum_{i = 0}^{q+1} e_i \otimes M(x_i = 1, x_j = 0 \text{ for } i \neq j),
\end{equation*}
The matrix $M(0, \dots, 0, 1)$ has rank 1, hence we can find $\lambda^{(1)}_0, \dots , \lambda^{(1)}_q$ using~\autoref{thm:aidedrankreduction} such that
\begin{equation*}
    T^{(1)}_{CW,q} = \sum_{i = 0}^{q} e_i \otimes \underbrace{\left(M(x_i = 1, x_j = 0 \text{ for } i \neq j) - \lambda^{(1)}_i M(0,\dots,0,1)\right)}_{ =\colon M^{(1)}(x_i = 1, x_j = 0 \text{ for } i \neq j)}
\end{equation*}
satisfies
\begin{equation*}
    R^{\thickdot p}(T^{(1)}_{CW,q}) = R^{\thickdot p}(T_{CW,q}) -1.
\end{equation*}
Note that the matrices $M^{(1)}(x_0, \dots x_q)$ have the form
\begin{equation*}
    M^{(1)}(x_0, \dots , x_{q}) = \begin{pmatrix}
    * &  x_1 &  \dots &  x_q&  x_0  \\
    x_1      &  x_0  &          &       &       \\
    \vdots   &       &  \ddots  &       &       \\
    x_q      &       &          &  x_0  &       \\
    x_0      &       &          &       &  0
    \end{pmatrix}.
\end{equation*}
Still, $M^{(1)}(0, \dots, 0, 1)$ has only non-zero entries in the first column or in the first row. That is, it has rank less than or equal to $p$, hence we can apply~\autoref{thm:aidedrankreduction} again and obtain $\lambda^{(2)}_0, \dots , \lambda^{(2)}_q$ such that 
\begin{equation*}
    T^{(2)}_{CW,q} = \sum_{i = 0}^{q-1} e_i \otimes \underbrace{\left(M^{(1)}(x_i = 1, x_j = 0 \text{ for } i \neq j) - \lambda^{(2)}_i M^{(1)}(0,\dots,0,1)\right)}_{ =\colon M^{(2)}(x_i = 1, x_j = 0 \text{ for } i \neq j)}
\end{equation*}
satisfies 
\begin{equation*}
    R^{\thickdot p}(T^{(2)}_{CW,q})= R^{\thickdot p}(T^{(1)}_{CW,q}) -1.
\end{equation*}
Again, we see that the elements of $T^{(2)}_{CW,q}(V_1^*)$ have the form
\begin{equation*}
    M^{(2)}(x_0, \dots , x_{q-1}) = \begin{pmatrix}
    *     &  x_1  &  \dots   &  *    &  x_0  \\
    x_1      &  x_0  &          &       &       \\
    \vdots   &       &  \ddots  &       &       \\
    *        &       &          &  x_0  &       \\
    x_0      &       &          &       &  0
    \end{pmatrix}
\end{equation*}
Repeating this procedure $q+2$ times leads to
\begin{equation*}
    T^{(q+1)}_{CW,q} = e_0 \otimes \underbrace{\left(M^{(q)}(1,0) - \lambda^{(q+1)} M^{(q)}(0,1)\right)}_{ =\colon M^{(q+1)}}.
\end{equation*}
with 
\begin{equation*}
    M^{(q+1)} = \begin{pmatrix}
    *     &  *  &  \dots   &  *    &  1  \\
    *      &  1  &          &       &       \\
    \vdots   &       &  \ddots  &       &       \\
    *        &       &          &  1  &       \\
    1     &       &          &       &  0
    \end{pmatrix}
\end{equation*}
By~\autoref{thm:aidedrankreduction} we reduced the aided rank by exactly 1 in each step yielding
\begin{equation*}
    R^{\thickdot p}(T^{(q+1)}_{CW,q}) = R^{\thickdot p}(T^{(q)}_{CW,q}) -1= \dots  = R^{\thickdot p}(T_{CW,q}) -(q + 1).
\end{equation*}
As $M^{(q+1)}$ has rank $q +2$ it follows $R^{\thickdot p}(T^{(q+1)}_{CW,q}) = \left\lceil \frac{q + 2}{p}  \right\rceil.$ and with that
\begin{equation*}
    R^{\thickdot p}(T_{CW,q}) = q+1 + \left\lceil \frac{q + 2}{p}  \right\rceil.
\end{equation*}
\end{proof}

We can also find the following upper bound on the aided rank of $T_{CW,q}^{\boxtimes 2}$.

\begin{proposition}
It holds that 
\begin{equation*}
R^{\thickdot p^2}(T_{CW,q}^{\boxtimes 2}) \leq q^2 + 4q + 3 + \left\lceil \frac{(q+2)^2}{p^2}\right\rceil.
\end{equation*}
    In particular, there are choices of $m$, $p$ and $q$ such that both $\langle m \rangle^{\thickdot p} \ngeq T_{CW,q}$ and $(\langle m \rangle^{\thickdot p})^{\boxtimes 2} \geq \left(T_{CW,q}\right)^{\boxtimes 2}$ hold.
\end{proposition}
\begin{proof}
Let us write
\begin{equation*}
N(x_0, \dots , x_{q+1}, y_0, \dots , y_{q+1}) = 
 \begin{pmatrix}
    x_{q+1}\cdot M(y)  &  x_1 \cdot M(y)  &  \dots   &  x_q \cdot M(y) &  x_0 \cdot M(y)  \\
    x_1 \cdot M(y)      &  x_0 \cdot M(y)  &          &       &       \\
    \vdots   &       &  \ddots  &       &       \\
    x_q   \cdot M(y)    &       &          &  x_0  \cdot M(y) &       \\
    x_0     \cdot M(y)  &       &          &       &  0
    \end{pmatrix}
\end{equation*}
where the matrices $M(y)$ are as in the proof of~\autoref{prop:aidedrankcoppersmith} given by
\begin{equation*}
    M(y) = M(y_0, \dots , y_{q+1}) = \begin{pmatrix}
    y_{q+1}  &  y_1  &  \dots   &  y_q  &  y_0  \\
    y_1      &  y_0  &          &       &       \\
    \vdots   &       &  \ddots  &       &       \\
    y_q      &       &          &  y_0  &       \\
    y_0      &       &          &       &  0
    \end{pmatrix}.
\end{equation*}
With this, we have
\begin{equation*}
	\left(T_{CW,q}\right)^{\boxtimes 2} = \sum_{i,j = 0}^{q+1} (e_{i}\otimes e_{j}) \otimes N(x_i = 1, y_j = 1).
\end{equation*}
and consequently, 
\begin{equation*}
\left(T_{CW,q}\right)^{\boxtimes 2} (\left(V_1^{\otimes 2}\right)^*) = \lbrace N(x,y): x, y \in \mathbb{C}^{q+2}\rbrace.
\end{equation*}
The rank of the matrix $N(x,y)$ depends on these vectors $x$ and $y$.
\begin{enumerate}[(i)]
\item If $x = y = e_0$, the matrix $N(x,y)$ has rank $(q+2)^2$.
\item If $x= e_0$ and $y = e_{q+1}$ or if $x = e_{q+1}$ and $y = e_0$, the matrix $N(x,y)$ has rank $q+2$.
\item If $x= e_0$ and $y \in \lbrace e_{1}, \dots , e_q \rbrace $ or if $x  \in \lbrace e_{1}, \dots , e_q \rbrace $ and $y = e_0$ the matrix $N(x,y)$ has rank $2(q+2)$.
\item  If $x= e_{q+1}$ and $y \in \lbrace e_{1}, \dots , e_q \rbrace $ or if $x  \in \lbrace e_{1}, \dots , e_q \rbrace $ and $y = e_{q+1}$ the matrix $N(x,y)$ has rank $2$.
\item If $x = y = e_{q+1}$, the matrix $N(x,y)$ has rank $1$.
\item If $x \in \lbrace e_{1}, \dots , e_q \rbrace $ and $y \in \lbrace e_{1}, \dots , e_q \rbrace $ the matrix $N(x,y)$ has rank $4$.
\end{enumerate}
Hence, to generate $\left(T_{CW,q}\right)^{\boxtimes 2} (A^*)$, we need to generate one matrix of rank $1$, $2q$ matrices of rank $2$, $q^2$ matrices of rank 4, $2$ matrices of rank $q+2$, $2q$ matrices of rank $2(q+2)$ and one matrix of rank $(q+2)^2$. Assuming $p^2 \geq 2(q+2)$, we will need at most
\begin{equation*}
q^2 + 4q + 3 + \left\lceil \frac{(q+2)^2}{p^2}\right\rceil
\end{equation*}
matrices of rank $p^2$ to generate $\left(T_{CW,q}\right)^{\boxtimes 2} (\left( V_1^{\otimes 2} \right)^*)$. In other words, 
\begin{equation}\label{eq:aidedrankofcwsquared}
R^{\thickdot p^2}(T_{CW,q}^{\boxtimes 2}) \leq q^2 + 4q + 3 + \left\lceil \frac{(q+2)^2}{p^2}\right\rceil
\end{equation}

To find $m$,$p$ and $q$ such that $\langle m \rangle^{\thickdot p} \ngeq T_{CW,q}$ but $\langle m^2 \rangle^{\thickdot p^2} \geq \left(T_{CW,q}\right)^{\boxtimes 2}$, we choose  $p$ and $q$ such that $\left\lceil \frac{q+2}{p+1}\right\rceil < \left\lceil \frac{q+2}{p}\right\rceil$ and $m = q+1+\left\lceil \frac{q+2}{p} \right\rceil$. By construction, we have $\langle m \rangle^{\thickdot p} \ngeq T_{CW,q}$.  We have found an example whenever
\begin{equation*}
q^2 + 4q + 3 + \left\lceil \frac{(q+2)^2}{p^2}\right\rceil \leq  \left(q+1+\left\lceil \frac{q+2}{p+1} \right\rceil \right)^2.
\end{equation*}
To see an explicit example, pick $q = 11$ and $p = 6$. We have 
\begin{align*}
R^{\thickdot 6}(T_{CW,11}) =  11 + 1 + \left\lceil \frac{11+2}{6} \right\rceil = 15 \\
R^{\thickdot 7}(T_{CW,11}) =  11 + 1 + \left\lceil \frac{11+2}{7} \right\rceil = 14,
\end{align*}
that is, $\langle 14 \rangle^{\thickdot 7} \geq T_{CW,11}$ but $\langle 14 \rangle^{\thickdot 6} \ngeq T_{CW,11}$.
From~\autoref{eq:aidedrankofcwsquared}, we get 
\begin{equation*}
R^{\thickdot 6^2}(T_{CW,11}^{\boxtimes 2}) \leq 11^2 + 4\cdot 11 + 3 + \left\lceil \frac{13^2}{6^2}\right\rceil = 173 \leq 14^2 = 196.
\end{equation*}
That gives
\begin{equation*}
\left( \langle 14 \rangle^{\thickdot 6} \right)^{\boxtimes 2} = \langle 196 \rangle^{\thickdot 36} \geq \langle 173 \rangle^{\thickdot 36} \geq T_{CW,11}^{\boxtimes 2}.
\end{equation*}
\end{proof}
\end{appendix}
\end{document}